\documentclass[11pt,reqno]{amsart}
\usepackage[T1]{fontenc}
\usepackage{graphicx, mathtools}
\usepackage{cite}

\usepackage{color}
\definecolor{MyLinkColor}{rgb}{0,0,0.4}

\newcommand{\R}{{\mathbb R}}
\newcommand{\C}{{\mathbb C}}

\newcommand{\bA}{{\mathbb A}}
\newcommand{\bB}{\mathbb{B}}

\newcommand{\N}{{\mathbb N}}

\newcommand{\cO}{\mathcal{O}}

\newcommand{\kL}{\mathcal{L}}

\newcommand{\wt}{\widetilde}
\newcommand{\oo}{\overline\omega}
\newcommand{\re}{\mathop{\rm Re}\nolimits}

\newcommand{\PV}{\mathop{\rm PV}\nolimits}
\newcommand{\ov}{\overline}
\newcommand{\p}{\partial}

\newcommand{\e}{\varepsilon}

\newcommand{\0}{\Omega}
\newcommand{\G}{\Gamma}

\newcommand{\supp}{\mathop{\rm supp}\nolimits}
\newtheorem{thm}{Theorem}[section]
\newtheorem{prop}[thm]{Proposition}
\newtheorem{lemma}[thm]{Lemma}
\newtheorem{cor}[thm]{Corollary}
\theoremstyle{remark}

\numberwithin{equation}{section} 

\title[The multiphase Muskat problem with equal viscosities in 2D]{The multiphase Muskat problem with equal viscosities in two dimensions}

\author{Jonas Bierler}
\address{Fakult\"at f\"ur Mathematik, Universit\"at Regensburg,   93040 Regensburg, Deutschland.}
\email{jonas.bierler@ur.de}
 \email{bogdan.matioc@ur.de}

\author{Bogdan--Vasile Matioc}


\usepackage[colorlinks=true,linkcolor=MyLinkColor,citecolor=MyLinkColor]{hyperref} 

\begin{document}

\begin{abstract}
We study the two-dimensional multiphase Muskat problem describing the motion of three immiscible fluids with equal viscosities in a vertical homogeneous porous 
 medium identified with $\R^2$ under the effect of gravity.
We first   formulate the  governing equations  as a strongly coupled evolution problem for the functions that parameterize the sharp interfaces between the fluids.  
Afterwards we prove that the problem is of parabolic type and establish its well-posedness together with two parabolic smoothing  properties.
For solutions that are not global we exclude, in a certain regime,  that the interfaces come into contact along a curve segment.
\end{abstract}

\maketitle

\section{Introduction and the main results}\label {Sec:1}
\subsection*{The mathematical model}
In this paper we study the two-dimensional multiphase Muskat problem describing the motion of three incompressible fluids with positive constant densities 
\[
\rho_3>\rho_2>\rho_1
\]
in a vertical porous medium identified with~$\R^2$.
Such three-phase flows are of great interest not only from a mathematical point of view but also in many other areas of science and technology, 
such as petroleum extraction and environmental engineering, cf. e.g. \cite{BT12, AP14}.  
In this paper we restrict our attention to the  particular case when the three fluids have  equal viscosities, which we denote by $\mu>0$.  
 We further assume that the fluid phases are separated at each time instant~$t{\geq0}$ by sharp interfaces which we described as being the graphs
 \[
\Gamma_{f}^{c_\infty}(t)\coloneqq{}\{(x,c_\infty+f(t,x))\,:\, x\in\R\} \quad\text{and}\quad  \Gamma_h(t)\coloneqq{}\{(x,h(t,x))\,:\, x\in\R\}.
 \] 
Here $c_\infty$ is a fixed positive constant and the functions $f(t),\, h(t):\R\to\R$ are unknown and are assumed to satisfy $f(t)+c_\infty>h(t)$ during the motion.
The fluid with density $\rho_i$ occupies the domain $\0_i(t)\subset\R^2$, $1\leq i\leq 3$, where $\0_2(t)\coloneqq{}\R^2\setminus\overline{\0_1(t)\cup\0_3(t)}$ and
\begin{align*}
\0_1(t)\coloneqq{}\{(x,y)\,:\, y>f(t,x)+c_\infty\},\quad \0_3(t)\coloneqq{}\{(x,y)\,:\, y<h(t,x)\}.
\end{align*}

 In the fluid layers the dynamic  is governed by the  equations 
 \begin{subequations}\label{PB}
\begin{equation}\label{eq:S1}
\left.\begin{array}{rllllll}
v_i(t)\!\!\!\!&=&\!\!\!\!-\cfrac{k}{\mu}\big(\nabla p_i(t)+(0,\rho_i g)\big)   , \\[1ex]
{\rm div}\,  v_i(t)\!\!\!\!&=&\!\!\!\!0
\end{array}
\right\}\qquad\text{in $ \0_i(t)$, $1\leq i\leq 3$},
\end{equation} 
  where $p_i(t)$ is the pressure and $v_i(t)\coloneqq{}(v_i^1(t),v_i^2(t))$ denotes the velocity field of the fluid located at $\0_i(t)$.
  The positive constant $k$ is the permeability of the  homogeneous porous medium and $g$ is the Earth's gravity. 
  The equation $\eqref{eq:S1}_1$ is Darcy's law which is oftenly used to model  flows in porous media, cf. e.g. \cite{Be88},  and $\eqref{eq:S1}_2$ describes the conservation of mass in all phases.
 
 We supplement \eqref{eq:S1} with  the following boundary conditions  at the free interfaces
\begin{equation}\label{eq:S2}
\left.\begin{array}{rllllll}
p_i(t)\!\!\!\!&=&\!\!\!\!p_{i+1}(t), \\[1ex]
 \langle v_i(t)| \nu_i(t)\rangle\!\!\!\!&=&\!\!\!\!  \langle v_{i+1}(t)| \nu_i(t)\rangle 
\end{array}
\right\}\qquad\text{on $\p\0_i(t)\cap\p\0_{i+1}(t)$, $i=1,\, 2$.}
\end{equation} 
Here    $\nu_i(t) $ denotes the unit normal at $\p\0_i(t)\cap\p\0_{i+1}(t)$ pointing into $\0_{i}(t)$ and~${\langle \, \cdot\,|\,\cdot\,\rangle}$ is the inner scalar product in~$\R^2$.
Additionally, we impose the  far-field boundary  conditions
 \begin{equation}\label{eq:S3}
\left.\begin{array}{rlllll}
v_i(t,x,y)\!\!\!\!&\to 0 \qquad\text{for  $|(x,y)|\to\infty$,  $1\leq i\leq 3$},\\[1ex]
f^2(t,x)+h^2(t,x)\!\!\!\!&\to 0 \qquad\text{for  $|x|\to\infty$,}
\end{array}\right\}
\end{equation}
which state that far away the flow  is nearly stationary.

Finally, in order to describe the motion of the free interfaces, we set their normal velocity equal to the normal component of the velocity field at the free boundary, that is
 \begin{equation}\label{eq:S4}
  \left.\begin{array}{rllllll}
 \p_tf(t)\!\!\!\!&=&\!\!\!\! \langle v_1(t)| (-\p_x f(t),1)\rangle \qquad\text{on   $\Gamma_{f}^{c_\infty}(t),$}\\[1ex]
  \p_th(t)\!\!\!\!&=&\!\!\!\! \langle v_2(t)| (-\p_x h(t),1)\rangle \qquad\text{on   $\G_h(t),$}
 \end{array}\right\}
\end{equation}  
and we  impose the initial condition
\begin{equation}\label{eq:S5}
(f,h)(0,\cdot )= (f_0,h_0).
\end{equation} 
\end{subequations}
We call the closed system  \eqref{PB}  the multiphase Muskat problem.\medskip

\subsection*{Summary of results and outline of the paper}
 The classical Muskat problem describing the dynamics of two fluid phases under the influence of gravity has received recently much attention in the mathematics community.
The numerous studies addressed the well-posedness issue  \cite{AN20x, MBV18, MBV19, AL20, BCS16, CCG11, CG07, NF20x, NP20},  
questions related to  global existence of solutions \cite{MBV19, Cam19, CCGS13, CGCSS14x, CGSV17, DLL17, GGPS19, PS17} 
and to  singularity formation \cite{CCFG13, CCFGL12}, but also the modeling  and dynamics of such flows in an inhomogeneous porous medium  \cite{BCG14, GG14, PCT17}.

For the multiphase Muskat problem \eqref{PB} considered herein much lees is known. 
This setting has been studied before only in a three-dimensional setting in \cite{CG10} where the authors
 established a local existence and uniqueness result in $H^k(\R^2)$ with $k\geq 4$.
 Moreover, it is shown in \cite{CG10} for solutions which are not global but bounded in $C^{1+\gamma}(\R^2)$, $\gamma\in(0,1)$, that the fluid interfaces cannot touch along a curve segment 
 when the time approaches the maximal  existence time, 
 excluding thus the occurrence of so-called squirt singularities.
A related scenario has been considered in two-dimensions, but in a periodic setting and with one of the fluids being air at uniform pressure, in \cite{EMM12a, EMW18} where
 the well-posedness and  the stability of equilibria are investigated.
 
Similarly as in the three dimensional case \cite{CG10}, we show herein that problem~\eqref{PB} can be expressed as a nonlinear, nonlocal, and 
strongly coupled  evolution problem with  nonlinearities described by contour integrals, cf.~\eqref{eq:F} below.
The equivalence of the formulations~\eqref{PB} and~\eqref{eq:F} is rigorously established in Theorem~\ref{T:1} below, by making use of the results from Appendix~\ref{Sec:A}.
The analysis in Appendix~\ref{Sec:A}, where in particular we extend Privalov's theorem to contour integrals over unbounded graphs (see Theorem~\ref{T:Priva}), 
 also motivates the choice of homogeneous Sobolev spaces in the study of the Muskat problem \cite{AL20}.
Our second main result stated in Theorem~\ref{MT1} establishes the well-posedness of the problem 
in the subcritical Sobolev spaces $H^s(\R)^2$ with $s\in(3/2,2)$. 
It also provides two parabolic smoothing properties.
Finally, in Proposition~\ref{MP1} we show for bounded solutions with finite existence time that the fluid interfaces intersect when the time approaches the maximal  existence time at least in one point, 
 but we exclude  also in this two-dimensional scenario the formation of squirt singularities.
 
 Compared to the two-phase Muskat  problem, cf. \cite{CG07}, new difficulties arise from the fact that the  coupling terms in \eqref{eq:F} are of highest order.
  However, based on the mapping properties established in Section~\ref{Sec:2}, we prove that the linearized operator, 
  which is represented as a $2\times 2$ matrix, see Section~\ref{Sec:3}, has lower order off-diagonal entries.
  A similar feature has been evinced for the Muskat problem investigated in  \cite{EMW18}. 
  The benefit of this weak coupling at the level of the linearization is that only the diagonal terms 
   need to be considered  when establishing  parabolicity for the problem.
   Once this is done, we can make use of  the abstract parabolic theory from \cite{L95} in the study of this multiphase Muskat problem.

\subsection*{Notation.} Given $k,\,n\in\N$ and an open set $\0\subset\R^n$, we denote by ${\rm C}^{k}(\0)$ the space consisting  of real-valued $k$-time continuously differentiable functions on $\0$,
and  ${\rm UC}^{k}(\0)$ is  the subspace of~${{\rm C}^{k}(\0)}$ having functions with uniformly continuous derivatives up to order  $k$ as elements. 
Moreover,~${{\rm BUC}^{k}(\0)}$ is the Banach space of functions with bounded and uniformly continuous derivatives up to order  $k$.
Finally, given $\alpha\in(0,1)$, we set 
\[
{\rm BUC}^{k+\alpha}(\0)\coloneqq{}\Big\{f\in {\rm BUC}^{k}(\0)\,:\, [\p^\beta f]_\alpha\coloneqq{}\sup_{x\neq y}\frac{|\p^\beta f(x)-\p^\beta f(y)|}{|x-y|^\alpha}<\infty \,\forall \, |\beta|=k\Big\}.
\]
Given Banach spaces $X$ and $Y$, the space ${\rm C}^{1-}(X,Y)$ consists of all locally Lipschitz maps from~$X$ to~$Y$.
Moreover,  we write $A\in \kL^k_{\rm sym}(X,Y)$  if~$A:\;X^k\to Y$ is $k$-linear, bounded, and symmetric.

\subsection*{Solving  the fixed time problem}
A remarkable property of~\eqref{PB} is the fact that the equations~\eqref{eq:S1}-\eqref{eq:S3}
are linear and have constant coefficients. 
This property  enables us to identify the velocity field in terms of the a priori unknown functions~$f$ and $h$ by means of contour integrals.
Such an approach has been  followed in the context of the Muskat problem at least at formal level already in the 80's, cf. \cite{DR84}.
For the clarity of the exposition we omit in this part the time dependence and  write $(\cdot)'$ for the the $x$-derivative of functions that depend only on~$x$.
In Theorem~\ref{T:1} below we provide, under suitable regularity constraints, an explicit formula for the velocity field in terms of~$X\coloneqq{}(f,h)$.
Our approach generalizes  the one followed in \cite{MBV19} in the context of the two-phase Muskat problem and strongly relies  on results from Appendix~\ref{Sec:A}.

\begin{thm}\label{T:1} Let $r\in(3/2,2)$, $c_\infty>0$, and $f,\,h\in H^r(\R)$ with $c_\infty+f>h$ be given.
Then the boundary value  problem
\begin{equation}\label{eq:Stat}
\left.\begin{array}{rllllll}
 v_i\!\!\!\!&=&\!\!\!\!-\displaystyle\frac{k}{\mu}\big(\nabla p_i+(0,\rho_i g)\big)&\text{in $ \0_i$, $1\leq i\leq 3$},\\[1ex]
{\rm div}\,  v_i\!\!\!\!&=&\!\!\!\!0 &\text{in $ \0_i$, $1\leq i\leq 3$}, \\[1ex]
p_i\!\!\!\!&=&\!\!\!\!p_{i+1}&\text{on $\p\0_i\cap\p\0_{i+1}$, $i=1,\, 2$}, \\[1ex]
 \langle v_i| \nu_i\rangle\!\!\!\!&=&\!\!\!\!  \langle v_{i+1}| \nu_i\rangle &\text{on $\p\0_i\cap\p\0_{i+1}$, $i=1,\, 2$,}\\[1ex]
v_i(x,y)\!\!\!\!&\to&\!\!\!\! 0 &\text{for  $|(x,y)|\to\infty$,  $1\leq i\leq 3$} 
\end{array}\right\}
\end{equation}
has a unique solution\footnote{The pressures $(p_1,p_2,p_3)$ are unique only up to the same additive constant.} $(v_1,v_2,v_3,p_1,p_2,p_3)$  with 
\[
v_i\in {\rm BUC}(\0_i)\cap {\rm C}^\infty(\0_i) \qquad\text{and}\qquad p_i\in {\rm UC}^1(\0_i)\cap {\rm C}^\infty(\0_i),\qquad 1\leq i\leq 3.
\]
Moreover,  setting $v\coloneqq{} v_1{\bf 1}_{\0_1}+v_2{\bf 1}_{\0_2}+v_3{\bf 1}_{\0_3}$, it holds for $z\coloneqq{}(x,y)\in\R^2\setminus(\G_h\cup\Gamma_{f}^{c_\infty})$ that
\begin{align}\label{forvelo}
v(z)=\frac{\Theta_1}{\pi}\int_\R\frac{(c_\infty+f(s)-y,x-s)}{(x-s)^2+(y-c_\infty-f(s))^2}f'(s)\, ds+\frac{\Theta_2}{\pi}\int_\R\frac{(h(s)-y,x-s)}{(x-s)^2+(y-h(s))^2}h'(s)\, ds,
\end{align}
with constants
\begin{align}\label{T1+T2}
\Theta_1\coloneqq{}\frac{kg(\rho_1-\rho_2)}{2\mu}\qquad\text{and}\qquad\Theta_2\coloneqq{} \frac{kg(\rho_2-\rho_3)}{2\mu}.
\end{align}
\end{thm}
\begin{proof} We devise the proof into two steps.\medskip

\noindent{Existence.} Let $v:\R^2\setminus(\G_h\cup\Gamma_{f}^{c_\infty})\to\R^2$ be given by \eqref{forvelo} and set $v_i\coloneqq{} v|_{\0_i}$, $1\leq i\leq 3$. 
In the notation from Appendix~\ref{Sec:A}, see \eqref{F:V}, it holds that 
\begin{align*}
v(z)=2\Theta_1v(f)[f']( z-(0,c_\infty))+2\Theta_2v(h)[h'](z),\qquad z\in \R^2\setminus(\G_h\cup\Gamma_{f}^{c_\infty}).
\end{align*}
Then $v\in {\rm C}^\infty(\R^2\setminus(\G_h\cup\Gamma_{f}^{c_\infty}))$ and, according to Theorem~\ref{T:Priva}, we also have~${v_i\in {\rm BUC}^{r-3/2}(\0_i)}$ for~$1\leq i\leq 3.$
Moreover, Lemma~\ref{L:vanis} yields  that $\eqref{eq:Stat}_2$ and~$\eqref{eq:Stat}_5$ hold true. 
In view of Lemma~\ref{L:Plemelj} we further get
\begin{align}
v_i(x,c_\infty+f(x))&=\frac{\Theta_1}{\pi}\PV\int_\R\frac{(-\delta_{[x,s]}f,s)}{s^2+(\delta_{[x,s]}f)^2}f'(x-s)\, ds
+\frac{\Theta_2}{\pi}\int_\R\frac{(-\delta_{[x,s]}X,s)}{s^2+(\delta_{[x,s]}X)^2}h'(x-s)\, ds\nonumber\\[1ex]
&\hspace{0.5cm}+(-1)^{i}\Theta_1\frac{f'(1,f')}{1+f'^2}(x),\quad i=1,\, 2,\label{v12rand}
\end{align}
and
\begin{align}
v_i(x,h(x))&=\frac{\Theta_1}{\pi}\int_\R\frac{(-\delta'_{[x,s]}X,s)}{s^2+(\delta'_{[x,s]}X)^2}f'(x-s)\, ds
+\frac{\Theta_2}{\pi}\PV\int_\R\frac{(-\delta_{[x,s]}h,s)}{s^2+(\delta_{[x,s]}h)^2} h'(x-s)\, ds\nonumber\\[1ex]
&\hspace{0.5cm}+(-1)^{i+1}\Theta_2\frac{h'(1,h')}{1+h'^2}(x),\quad i=2,\, 3,\label{v23rand}
\end{align}
where $\PV$ is the principal value and, setting $X\coloneqq{}(f,h)$, we defined 
\begin{equation}\label{notat}
\begin{aligned}
\delta_{[x,s]}f&\coloneqq{} f(x)-f(x-s),\\[1ex]
 \delta_{[x,s]}X&\coloneqq{} c_\infty+f(x)-h(x-s),\\[1ex]
  \delta'_{[x,s]}X&\coloneqq{} h(x)-c_\infty-f(x-s).
\end{aligned}
\end{equation}
The formulas \eqref{v12rand} and \eqref{v23rand} now show the validity of $\eqref{eq:Stat}_4$.

We next define pressures $p_i:\0_i\to\R$, $1\leq i\leq 3,$ by the formula
\begin{align}\label{pressures}
p_i(x,y)&\coloneqq{}-\frac{k}{\mu}\Big(\int_0^x\langle v_i(s,d_i(s))|(1,d_i'(s))\rangle\, ds+\int_{d_i(x)}^yv_i^2(x,s)\, ds\Big)-\rho_igy+c_i,
\end{align}
where $v_i\eqqcolon(v_i^1,v_i^2)$, $c_i\in\R$ are constants, and  with
\[
d_1\coloneqq{}\|f\|_\infty+c_\infty+1,\qquad d_2\coloneqq{}\frac{1}{2}(c_\infty+f+h),\qquad d_3\coloneqq{}-\|h\|_\infty-1. 
\]
Taking advantage of $\p_yv^1=\p_xv^2$ in  $\R^2\setminus(\G_h\cup\Gamma_{f}^{c_\infty})$, cf. Lemma~\ref{L:vanis}, we deduce that~${p_i\in{\rm  C}^1(\0_i)}$  and that~$\eqref{eq:Stat}_1$ is satisfied.
The regularity properties established  for $v_i$ now imply that indeed~${p_i\in {\rm UC}^1(\0_i)\cap {\rm C}^\infty(\0_i)}$, $1\leq i\leq 3$.
Using $\eqref{eq:Stat}_1$ and \eqref{v12rand}-\eqref{v23rand}, it then immediately follows that~$(p_1-p_2)|_{\G_f^{c_\infty}}$ and~$(p_2-p_3)|_{\G_h}$ are constants.
 Hence, for a suitable choice of~$c_i$, we may achieve that $\eqref{eq:Stat}_3$ are satisfied. 
 Therewith we established the existence of at least a solution to~\eqref{eq:Stat}.\medskip
 
\noindent{Uniqueness.} It remains to show that the system \eqref{eq:Stat} has, when setting the gravity constant~$g$ equal to zero, only the trivial solutions defined by $v=(v^1,v^2)=0$ and $p=c\in\R$.
 To begin, we note that $\eqref{eq:Stat}_1$ implies $\p_yv_i^1-\p_x v_i^2=0$ in $\0_i,$ $1\leq i\leq 3$. 
 Moreover, combining~$\eqref{eq:Stat}_1$,~$ \eqref{eq:Stat}_3$,  and~$\eqref{eq:Stat}_4$, we get
 $v\in{\rm BUC}(\R^2)$.
 Stokes' theorem then yields 
 \begin{align}\label{rot0}
 \p_y v^1-\p_x v^2=0\qquad\text{in $\mathcal{D}'(\R^2)$}.
\end{align}  
We next set  $\Psi\coloneqq{} \psi_1{\bf 1}_{\ov{\0_1}}+\psi_2{\bf 1}_{\ov{\0_2}}+\psi_3{\bf 1}_{\ov{\0_3}}$, where
$\psi_i:\ov{\0_i}\to\R$ are given  by
\begin{align*}
\psi_i(x,y)&\coloneqq{}\int_{h(x)}^yv^1(x,s)\, ds-\int_0^x\langle v(s,h(s))|(-h'(s),1)\rangle\, ds,\qquad i=2,\, 3,\\[1ex]
\psi_1(x,y)&\coloneqq{}\int_{c_\infty+f(x)}^yv^1(x,s)\, ds+\psi_2(x,c_\infty+f(x)).
\end{align*}
It follows immediately that  $\Psi\in {\rm C}(\R^2)$.
 Additionally, using Stokes's theorem and $\eqref{eq:Stat}_2$ we may show that $\nabla\psi_i=(-v^2,v^1)$ in $\mathcal{D}'(\0_i)$, $1\leq i\leq 3$. 
As a direct consequence we get~${\psi_i\in {\rm UC}^1(\0_i)}$ for  $1\leq i\leq 3$.
Additionally,~${\nabla\Psi\in\mathcal{D}'(\R^2)}$ belongs to ${\rm BUC}(\R^2)$, hence~${\Psi\in {\rm UC}^1(\R^2)}$.
Therefore, given~${\varphi\in{\rm C}^\infty_0(\R^2),}$ we have
\begin{align*}
\langle\Delta\Psi,\varphi\rangle=\int_{\R^2} \Psi\Delta\varphi\, dz=-\int_{\R^2}\langle\nabla\Psi|\nabla\varphi\rangle\, dz=\int_{\R^2}\langle( v^2,-v^1)|\nabla\varphi\rangle\, dz=\langle \p_yv^1-\p_x v^2,\varphi\rangle,
\end{align*}
and \eqref{rot0} then yields  $\Delta\Psi=0$ in $\mathcal{D}'(\R^2)$.
Consequently,~$\Psi$ is the real part of a holomorphic function~${u:\C\to\C}$.
Since $u'$ is  holomorphic too and ${u'=(\p_x\Psi,-\p_y\Psi)=-(v^2,v^1)}$ is bounded, cf. $\eqref{eq:Stat}_5$, Liouville's theorem  yields $u'=0$, hence $v=0$.
Moreover, in view of $\eqref{eq:Stat}_1$, we now obtain that  $\nabla p=0$ in $\R^2$, meaning that $p$ is constant in $\R^2$. This completes our arguments.
\end{proof}

\subsection*{The contour integral formulation and the main results}
Concerning the multiphase Muskat problem \eqref{PB}, Theorem~\ref{T:1} implies 
that if, at any given time $t\geq0$, $f(t) $ and $h(t)$ belong to $H^r(\R)$, with $r\in(3/2,2)$, and $c_\infty+f(t)>h(t)$, then $v_1(t)|_{\Gamma_f^{c_\infty}(t)}$ and $v_2(t)|_{\Gamma_h(t)}$
are given by \eqref{v12rand} and~\eqref{v23rand}.
Recalling also  \eqref{eq:S4},
we can thus formulate the moving boundary  problem \eqref{PB} as an autonomous evolution problem for the pair $X\coloneqq{}(f,h)$ which reads as
\begin{align}\label{eq:F}
\frac{d X(t)}{dt}=\Phi(X(t)),\quad t\geq 0,\qquad X(0,\cdot)=X_0\coloneqq{}(f_0,h_0),
\end{align}
where the nonlinear operator $\Phi=(\Phi_1,\Phi_2)$ is defined by
\begin{equation}\label{Phi1}
\Phi_1(X) \coloneqq \Theta_1\bB(f)[f']+\frac{\Theta_2}{\pi}\big((c_\infty+f)f'C_1(X)[h']-f'C_1(X)[hh']+D_1(X)[h']\big)
\end{equation}
and 
\begin{equation}\label{Phi2}
\Phi_2(X)\coloneqq\Theta_2\bB(h)[h']+\frac{\Theta_1}{\pi}\big((h-c_\infty)h'C_1'(X)[f']-h'C_1'(X)[ff']+D_1'(X)[f']\big).
\end{equation} 
The constants $\Theta_i$, $i=1,\, 2$, are introduced in \eqref{T1+T2} and, given $u\in H^{r}(\R)$, with $r\in(3/2,2)$ which is fixed in the remaining part, we denote by   $\bB(u)$ the linear operator 
\begin{align}
\bB(u)\coloneqq\frac{1}{\pi}\big( B_{0,1}^0(u) +u'B_{1,1}^0(u)\big),
\end{align}  
where the operators 
$B_{m,1}^0,$~${m=0,\,1}$, as well as~$C_1,\, C_1',\, D_1,$ and $ D_1'$ are defined in \eqref{BNM}-\eqref{Bom} and~\eqref{CDms} below.
We shall treat \eqref{eq:F} as a fully nonlinear evolution problem in $H^{r-1}(\R)^2$.
To this end we prove in Corollary~\ref{Cor:1} below that $\Phi$ is smooth, that is
\begin{align}\label{RegP}
\Phi\in{\rm C}^{\infty}(\cO_r,H^{r-1}(\R)^2),
\end{align}
where
\begin{align}\label{or}
\mathcal{O}_r\coloneqq{}\{(f,h)\in H^r(\R)^2\,:\, c_\infty+f>h\}.
\end{align}
Moreover, our analysis (see Proposition~\ref{P:pfo1} below) will disclose that  \eqref{eq:F}  is of parabolic type in the phase space $\cO_r$, that is the Fr\'echet derivative
 $\p\Phi(X)$ at any $X\in\cO_r$, viewed as an unbounded  operator in $H^{r-1}(\R)^2$ with domain $H^{r}(\R)^2$,
 is the generator of an analytic semigroup in $\kL(H^{r-1}(\R)^2)$. In the notation introduced in \cite{Am95} the latter property  writes as
 \begin{align}\label{Gen}
 -\p\Phi(X)\in\mathcal{H}(H^{r}(\R)^2,H^{r-1}(\R)^2).
\end{align}  
The properties \eqref{RegP} and \eqref{Gen} enable us to use the  parabolic theory presented in \cite{L95}  to
establish the following results for \eqref{eq:F}.

\begin{thm}\label{MT1}
Let $r\in(3/2,2)$.
Given $X_0\in\cO_r$, the multiphase Muskat problem  \eqref{eq:F} has a unique maximal solution~${X\coloneqq{} X(\,\cdot\,; X_0)}$  such that
\[ X\in {\rm C}([0,T^+),\cO_r)\cap {\rm C}^1([0,T^+),  H^{r-1}(\R)^2), \] 
with $T^+=T^+(X_0)\in(0,\infty]$ denoting the maximal time of existence.
Moreover, we have:
\begin{itemize}
\item[(i)] The solution depends continuously on the initial data;
\item[(ii)] Given $k\in\N$, we have $X\in {\rm C}^\infty((0,T^+ )\times\R,\R^2)\cap {\rm C}^\infty ((0,T^+),  H^k(\R)^2);$
 \item[(iii)] If $T^+<\infty$, then 
\[
 \sup_{t\in [0,T^+)}\|X(t)\|_{H^r(\R)}=\infty\qquad\text{or}\qquad \liminf_{t\to T^+}{\rm dist\,}(\G_f^{c_\infty}(t),\G_h(t))=0. 
\]
\end{itemize}
 \end{thm}

The next result  shows, for bounded solutions with  $T^+<\infty$, that the fluid interfaces intersect in at least one point along a sequence $t_n\to T^+$. 
Moreover, using the same strategy as in \cite{CFL04, CG10},  we exclude for such solutions that the two fluid interfaces collapse along a curve segment.
 
 \begin{prop}\label{MP1}
Let $ X\in {\rm C}([0,T^+),\cO_r)\cap {\rm C}^1([0,T^+),  H^{r-1}(\R)^2)$ be a maximal solution to~\eqref{eq:F} with  $T^+<\infty$ and such that, for some $M>0$, $\|X(t)\|_{H^r}\leq M$ for all $t\in[0,T^+)$. 
Then, there exists $x_0\in\R$  with  the property that
\begin{equation}\label{pxo}
\liminf_{t\to T^+}\,(c_\infty+f(t,x_0)-h(t,x_0))=0. 
\end{equation}
Moreover, for each $x_0$ satisfying \eqref{pxo} and for each $\delta>0$, we have  
\[
 \liminf_{t\to T^+}\sup_{\{|x-x_0|\leq \delta\}}(c_\infty+f(t,x)-h(t,x))>0. 
\]
 \end{prop}
The proofs of Theorem~\ref{MT1} and Proposition~\ref{MP1} are postponed to the very end of Section~\ref{Sec:3}.

\section{Mapping properties }\label{Sec:2} 

In this section we introduce the operators $B_{m,1}^0,$ ${m=0,\,1}$, and $C_1,\, C_1',\, D_1,\, D_1'$ which appear in  \eqref{Phi1}-\eqref{Phi2} in a more general context and study the properties of these operators. 
The main goal is to establish the smoothness property \eqref{RegP}, see Corollary~\ref{Cor:1} below.

Motivated by the formulas \eqref{v12rand} and \eqref{v23rand}, we   introduce the family $B_{n,m}$, $n,\, m\in\N$, of singular integral   operators, where,  given
  Lipschitz continuous  maps ${u_1,\ldots, u_{m},\, v_1, \ldots, v_n:\mathbb{R}\to\mathbb{R}}$ and~${\oo\in L_2(\R)}$,   the operator $B_{n,m}$ is defined by  
\begin{equation}\label{BNM}
B_{n,m}(u_1,\ldots, u_m)[v_1,\ldots,v_n,\oo](x)\coloneqq{}\PV\int_\mathbb{R} 
\cfrac{\prod_{i=1}^{n}\big(\delta_{[x,s]} u_i / s\big)}{\prod_{i=1}^{m}\big[1+\big(\delta_{[ x, s]}  v_i / s\big)^2\big]} \frac{\oo(x- s)}{ s}\, d s.
\end{equation}
Here  we use the notation introduced in \eqref{notat}.
Furthermore, we define
\begin{equation}\label{Bom}
B^0_{n,m}(u)[\oo]\coloneqq{} B_{n,m}(u,\ldots,u)[u,\ldots,u,\oo]
\end{equation}
The operators $B_{n,m}$ were introduced in \cite{MBV19}, but they are also important in the study of the Stokes problem, cf.~\cite{MP2021}.
It is important to point out that $B_{0,0}=\pi H,$ where $H$ denotes the Hilbert transform.
We now recall some important properties of these operators.
\begin{lemma}\label{L:MP0}
Let  $r\in(3/2 ,2)$ be fixed. 
\begin{itemize}
\item[(i)] Given  $u_1,\ldots, u_m\in H^r(\mathbb{R})$, there exists a constant $C$ that 
  depends only on $n,\, m,\, r$,  and~${\max_{1\leq i\leq m}\|u_i\|_{H^r}}$ such that
\begin{equation*} 
\| B_{n,m}(u_1,\ldots, u_{m})[v_1,\ldots, v_n,\oo]\|_{H^{r-1}}\leq C \|\oo\|_{H^{r-1}}\prod_{i=1}^{n}\|v_i\|_{H^{r}}
\end{equation*}
for all $v_1,\ldots, v_n\in H^r(\mathbb{R})$ and $\oo\in H^{r-1}(\mathbb{R}). $\\[-1ex]
\item[(ii)] The mapping $\big[u\mapsto B^0_{n,m}(u)\big]:  H^r(\R)\to\kL(H^{r-1}(\mathbb{R}))$ is smooth.
\end{itemize}
\end{lemma}
\begin{proof}
The claim (i) is established   \cite[Lemmas~2.5]{AM21x} an (ii) is proven in   \cite[Corollary C.5]{MP2021}.
\end{proof}

The evolution equation \eqref{eq:F} consists actually of an equation for $f$ and one for $h$ which are coupled. 
The coupling terms contain highest (first) order derivatives of both variables and they are expressed by using the aforementioned operators $C_1,\, C_1',\, D_1,\, D_1'$. 
We next introduce these operators as elements of a larger family of operators enjoying similar properties.

Given $1\leq m\in\N$ and $X_i\coloneqq{}(f_i,h_i)\in \cO_r$, $1\leq i\leq m$, we set
\begin{equation}\label{CDms}
\begin{aligned}
C_m(X_1,\ldots,X_m)[\oo](x)&\coloneqq{}\int_\R\frac{\oo(x-s)}{\prod_{i=1}^{m}\big[s^2+ (\delta_{[ x, s]}X_i)^2\big]}\, ds,\\[1ex]
 C'_m(X_1,\ldots,X_m)[\oo](x)&\coloneqq{}\int_\R\frac{\oo(x-s)}{\prod_{i=1}^{m}\big[s^2+ (\delta'_{[ x, s]}X_i)^2\big]}\, ds,\\[1ex]
 D_m(X_1,\ldots,X_m)[\oo](x)&\coloneqq{}\int_\R\frac{s\oo(x-s)}{\prod_{i=1}^{m}\big[s^2+ (\delta_{[ x, s]}X_i)^2\big]}\, ds,\\[1ex]
 D'_m(X_1,\ldots,X_m)[\oo](x)&\coloneqq{}\int_\R\frac{s\oo(x-s)}{\prod_{i=1}^{m}\big[s^2+ (\delta'_{[ x, s]}X_i)^2\big]}\, ds
\end{aligned}
\end{equation}
for $\oo\in L_2(\R)$ and $x\in\R$, where we used again the notation introduced in \eqref{notat}.
Since~$X_i\in\cO_r$ for~$1\leq i\leq m$, these operators are no longer singular. 
However,  the kernels of $D_1$ and~$D_1'$ behave  for large $s$ similarly as that of the truncated Hilbert transform~${H_\delta:L_2(\R)\to L_2(\R)}$, with $\delta>0$, 
which is defined by
\begin{align}\label{TrHi}
H_\delta [\oo](x)\coloneqq{}\frac{1}{\pi}\int_{\{|s|>\delta\}}\frac{\oo(x-s)}{s}\, ds,\qquad x\in\R.
\end{align}
We recall that, given $\delta>0$, $H_\delta$ is a Fourier multiplier with symbol $[\xi\mapsto m_\delta(\xi)]$ given by
\[
m_\delta(\xi)\coloneqq{}-\frac{2}{\pi} i{\rm sign\,}(\xi)\int_{\delta|\xi|}^\infty\frac{\sin(t)}{t}\, dt.
\]
Since $\|m_\delta\|_\infty\leq 2$ for all $\delta>0$, it follows that 
\begin{align}\label{H96}
\|H_\delta\|_{\kL(L_2(\R))}\leq 2.
\end{align}

We next study the mapping properties of the operators  $C_m,\, C_m',\, D_m,\, D_m'.$

\begin{lemma}\label{L:MP1a}
Given $1\leq m\in\N$ and $X_i\coloneqq{}(f_i,h_i)\in \cO_r$, $1\leq i\leq m$, we set
\begin{align}\label{c0}
c_0\coloneqq{}\min_{1\leq i\leq m}{\rm dist\,} (\G^{c_\infty}_{f_i},\G_{h_i}).
\end{align}
 Then, there exists a  constant $C$ that depends only on~$m$, $c_0$, and~$\max_{1\leq i\leq m}\|h_i'\|_\infty$ such that  
\begin{align}\label{CM1a}
\|C_m(X_1,\ldots,X_m)[\oo]\|_2+\|C_m'(X_1,\ldots,X_m)[\oo]\|_2\leq C\|\oo\|_2,\qquad \oo\in L_2(\R).
\end{align} 
\end{lemma}
\begin{proof}
Let
\begin{equation}\label{delta}
\delta\coloneqq{}\frac{c_0}{2(\max_{1\leq i\leq m}\|h_i'\|_\infty+1)}.
\end{equation}
Given $x\in\R$ and $|s|<\delta$, it holds
\begin{align}\label{estima}
\min\{ |\delta_{[ x, s]}X_i|,\, |\delta'_{[ x, s]}X_i|\}\geq {\rm dist\,} (\G^{c_\infty}_{f_i},\G_{h_i})-|\delta_{[ x, s]}h_i|\geq c_0/2,\qquad 1\leq i\leq m.
\end{align}
Therefore, when considering  $C_m$ (the case  $C_m'$ is similar), we get, by making use of Minkowski's integral inequality
\begin{align*}
\|C_m(X_1,\ldots,X_m)[\oo]\|_2&=\Big(\int_\R\Big|\int_\R\frac{\oo(x-s)}{\prod_{i=1}^{m}\big[s^2+ (\delta_{[ x, s]}X_i)^2\big]}\, ds\Big|^2\, dx\Big)^{1/2}\\[1ex]
&\leq\int_\R\Big(\int_\R\Big|\frac{\oo(x-s)}{\prod_{i=1}^{m}\big[s^2+ (\delta_{[ x, s]}X_i)^2\big]}\Big|^2\, dx\Big)^{1/2}\, ds\\[1ex]
&\leq\Big(\frac{2}{c_0}\Big)^{2m}\int_{\{|s|<\delta\}}\|\oo\|_2\, ds+\int_{\{|s|>\delta\}}\frac{\|\oo\|_2}{s^{2m}}\, ds\\[1ex]
&\leq C\|\oo\|_2
\end{align*}
and \eqref{CM1a} follows.
\end{proof}

It is not difficult to extend the proof of Lemma~\ref{L:MP1a} in the context of the operators $D_m$ and~$D_m'$ with~$m\geq 2$. 
The case $m=1$ is however more subtle  and requires a different strategy  which uses the estimate~\eqref{H96}.

\begin{lemma}\label{L:MP2a}
Given $1\leq m\in\N$ and $X_i\coloneqq{}(f_i,h_i)\in \cO_r$, $1\leq i\leq m$, let $c_0>0$ be the constant defined in~\eqref{c0}. 
Then,  there exists a  constant $C$ that depends only on~$r,$ $m$, $c_0$, and~$\max_{1\leq i\leq m}\|X_i\|_{H^r}$ such that  
\begin{align}\label{CM2a}
\|D_m(X_1,\ldots,X_m)[\oo]\|_2+\|D_m'(X_1,\ldots,X_m)[\oo]\|_2\leq C\|\oo\|_2,\qquad \oo\in L_2(\R),
\end{align} 
\end{lemma}
\begin{proof}
The proof in the case $m\geq 2$ follows along the lines of the proof of Lemma~\ref{L:MP1a}.
We now consider the operator $D_1$ (the estimate for $D_1'$ follows similarly).
Let~$\delta>0$ be as defined in~\eqref{delta} (with $m=1$) and   set 
\[
I(x,s)\coloneqq{}\frac{s\oo(x-s)}{ s^2+ (\delta_{[ x, s]}X_1)^2}=\Big(1-\frac{(\delta_{[ x, s]}X_1/s)^2}{ 1+ (\delta_{[ x, s]}X_1/s)^2}\Big)\frac{\oo(x-s)}{s},\qquad x,\,s\in\R,\, s\neq0.
\]
With $H_\delta$ denoting the truncated Hilbert transform, see \eqref{TrHi}, we have for $x\in\R$ that
\begin{align*}
|D_1(X_1)[\oo](x)|\leq \int_{\{|s|<\delta\}}|I(x,s)|\, ds+|H_\delta[\oo](x)|
+\int_{\{\delta<|s|\}}\frac{(\delta_{[ x, s]}X_1/s)^2}{ 1+ (\delta_{[ x, s]}X_1/s)^2}\Big|\frac{\oo(x-s)}{s}\Big|\, ds.
\end{align*}
Given $|s|<\delta,$ \eqref{estima} implies $|\delta_{[ x, s]}X_1|\geq c_0/2 $ and Minkowski's integral inequality then yields
\begin{align*}
\Big\|\int_{\{|s|<\delta\}}|I(\cdot,s)|\, ds\Big\|_2\leq \int_{\{|s|<\delta\}}\Big(\int_\R|I(x,s)|^2\, dx\Big)^{1/2}\, ds\leq  \frac{8\delta^2}{c_0^2} \|\oo\|_2.
\end{align*}
Moreover, taking into account that $|\delta_{[ \cdot, s]}X_1|\leq c_\infty+\|f_1\|_\infty+\|h_1\|_\infty$, Minkowski's integral inequality leads us to 
\begin{align*}
\Big\|\int_{\{\delta<|s|\}}\frac{(\delta_{[ \cdot, s]}X_1/s)^2}{ 1+ (\delta_{[ \cdot, s]}X_1/s)^2}\Big|\frac{\oo(\cdot-s)}{s}\Big|\, ds\Big\|_2
\leq (c_\infty+\|f_1\|_\infty+\|h_1\|_\infty)\int_{\{\delta<|s|\}}\frac{\|\oo\|_2}{s^2}\, ds\leq C\|\oo\|_2.
\end{align*}
Recalling \eqref{H96}, we conclude that \eqref{CM2a} is satisfied. 
\end{proof}

As a consequence of Lemma~\ref{L:MP1a} and Lemma~\ref{L:MP2a} we obtain the following result.
\begin{cor}\label{R:Kor} Given $1\leq m\in\N$, it holds that
\begin{align}\label{LLc1}
C_m,\, D_m,\,C_m',\, D_m'\in {\rm C}^{1-}(\cO_r^m,\kL(L_2(\R))).
\end{align}
\end{cor}
\begin{proof}
Let   $X_i=(f_i,h_i)$, $\wt X_i=(\wt f_i,\wt h_i)\in \cO_r$, $1\leq i\leq m$,  $\oo\in L_2(\R),$  and~$E_m\in \{C_m,\, D_m\}.$ 
It then follows
\begin{equation}\label{LLC}
\begin{aligned}
&\hspace{-0.5cm}E_m(X_1,\ldots,X_m)[\oo]-E_m(\wt X_1,\ldots,\wt X_m)[\oo]\\[1ex]
&=\sum_{j=1}^m\Big( (2c_\infty+\wt f_j+f_j)(\wt f_j-f_j)E_{m+1}(\wt X_1,\ldots, \wt X_j,X_j,\ldots,X_m)[\oo]\\[-1ex]
&\hspace{1.35cm}-(\wt f_j-f_j)E_{m+1}(\wt X_1,\ldots, \wt X_j,X_j,\ldots,X_m)[(\wt h_j+h_j)\oo]\\[1ex]
&\hspace{1.35cm}-(2c_\infty+\wt f_j+f_j)E_{m+1}(\wt X_1,\ldots, \wt X_j,X_j,\ldots,X_m)[(\wt h_j-h_j)\oo]\\[1ex]
&\hspace{1.35cm}+ E_{m+1}(\wt X_1,\ldots, \wt X_j,X_j,\ldots,X_m)[(\wt h_j^2-h_j^2)\oo]\Big),
\end{aligned}
\end{equation}
respectively
\begin{equation}\label{LLC'}
\begin{aligned}
&\hspace{-0.5cm}E'_m(X_1,\ldots,X_m)[\oo]-E'_m(\wt X_1,\ldots,\wt X_m)[\oo]\\[1ex]
&=\sum_{j=1}^m\Big( (\wt h_j^2-h_j^2)E'_{m+1}(\wt X_1,\ldots, \wt X_j,X_j,\ldots,X_m)[\oo]\\[-1ex]
&\hspace{1.35cm}-(\wt h_j-h_j)E'_{m+1}(\wt X_1,\ldots, \wt X_j,X_j,\ldots,X_m)[(2c_\infty+\wt f_j+f_j)\oo]\\[1ex]
&\hspace{1.35cm}-(\wt h_j+h_j)E'_{m+1}(\wt X_1,\ldots, \wt X_j,X_j,\ldots,X_m)[(\wt f_j-f_j)\oo]\\[1ex]
&\hspace{1.35cm}+ E'_{m+1}(\wt X_1,\ldots, \wt X_j,X_j,\ldots,X_m)[(2c_\infty+\wt f_j+f_j)(\wt f_j-f_j)\oo]\Big).
\end{aligned}
\end{equation}
Combining \eqref{LLC}, \eqref{LLC'}, Lemma~\ref{L:MP1a}, and Lemma~\ref{L:MP2a}  we conclude that \eqref{LLc1} holds true.
\end{proof}

We prove next that the operators $C_m,\, C_m',\, D_m,\, D_m'$ map, for given~${X_i\in \cO_r}$, $1\leq i\leq m$, the definition domain $L_2(\R)$  actually into $H^1(\R).$
\begin{lemma}\label{L:MP12b}
Given $1\leq m\in\N$ and $X_i\coloneqq{}(f_i,h_i)\in \cO_r$, $1\leq i\leq m$, let $c_0>0$ be the constant defined in~\eqref{c0}. 
Given $E_m\in \{C_m,\, D_m\},$  there exists a positive constant $C$ that depends only on~$r,$ $m$, $c_0$, and~$\max_{1\leq i\leq m}\|X_i\|_{H^r}$ such that  
\begin{align}\label{CM12b}
\|E_m(X_1,\ldots,X_m)[\oo]\|_{H^1}+\|E_m'(X_1,\ldots,X_m)[\oo]\|_{H^1}\leq C\|\oo\|_2,\qquad \oo\in L_2(\R).
\end{align} 
Moreover, $C_m,\, D_m,\,C_m',\, D_m'\in {\rm C}^{1-}(\cO_r^m,\kL(L_2(\R), H^1(\R))).$ 
\end{lemma}
\begin{proof}
Let $\{\tau_\xi\}_{\xi\in\R}$ denote the group of right translations and assume first that $\oo\in {\rm C}^\infty_0(\R)$.
 Given $0\neq\xi\in\R$, the formula~\eqref{LLC} leads us to
\begin{align*}
&\hspace{-0.5cm}\frac{E_m(X_1,\ldots,X_m)[\oo]-\tau_\xi(E_m(X_1,\ldots,X_m)[\oo])}{\xi}\\[1ex]
&=E_m(X_1,\ldots,X_m)\Big[\frac{\oo-\tau_\xi\oo}{\xi}\Big]\\[1ex]
&\hspace{0,45cm}+\sum_{j=1}^m\Big( (2c_\infty+\tau_\xi f_j+f_j)\frac{\tau_\xi f_j-f_j}{\xi}E_{m+1}(\tau_\xi X_1,\ldots, \tau_\xi X_j,X_j,\ldots,X_m)[\tau_\xi\oo]\\[1ex]
&\hspace{1.8cm}-\frac{\tau_\xi f_j-f_j}{\xi}E_{m+1}(\wt X_1,\ldots, \wt X_j,X_j,\ldots,X_m)[(\tau_\xi h_j+h_j)\tau_\xi\oo]\\[1ex]
&\hspace{1.8cm}-(2c_\infty+\tau_\xi f_j+f_j)E_{m+1}(\tau_\xi X_1,\ldots, \tau_\xi X_j,X_j,\ldots,X_m)\Big[\frac{\tau_\xi h_j-h_j}{\xi}\tau_\xi\oo\Big]\\[1ex]
&\hspace{1.8cm}+ E_{m+1}(\tau_\xi X_1,\ldots, \tau_\xi X_j,X_j,\ldots,X_m)\Big[\frac{(\tau_\xi h_j)^2-h_j^2}{\xi}\tau_\xi\oo\Big]\Big).
\end{align*}
Recalling~\eqref{LLc1}, we may pass to the limit $\xi\to0$ on the right of the latter equation and conclude that $E_m(X_1,\ldots,X_m)[\oo]\in H^1(\R)$, with 
\begin{align*}
&\hspace{-0.5cm}(E_m(X_1,\ldots,X_m)[\oo])'\\[1ex]
&=E_m(X_1,\ldots,X_m)[\oo']\\[1ex]
&\hspace{0,45cm}-2\sum_{j=1}^m\big( (c_\infty+ f_j)f_j'E_{m+1}( X_1,\ldots,X_m, X_j)[\oo]-f_j'E_{m+1}(X_1,\ldots,X_m, X_j)[h_j\oo]\\[1ex]
&\hspace{1.8cm}-( c_\infty+f_j)E_{m+1}( X_1,\ldots, X_m,X_j)[h_j'\oo]+ E_{m+1}(  X_1,\ldots ,X_m, X_j)[h_jh_j'\oo]\big).
\end{align*}
All the terms on the right are well-defined (and belong to $L_2(\R)$) when merely $\oo\in L_2(\R)$ except for  $E_m(X_1,\ldots,X_m)[\oo']$.
However, using integration by parts we can rewrite this term as
\begin{align*}
&\hspace{-0.5cm}C_m(X_1,\ldots,X_m)[\oo']\\[1ex]
&=-2\sum_{j=1}^m\big( D_{m+1}( X_1,\ldots,X_m, X_j)[\oo] +(c_\infty+ f_j)C_{m+1}( X_1,\ldots,X_m, X_j)[h_j'\oo]\\[1ex]
&\hspace{1.9cm} -C_{m+1}(  X_1,\ldots ,X_m, X_j)[h_jh_j'\oo]\big),
\end{align*}
respectively
\begin{align*}
&\hspace{-0.5cm}D_m(X_1,\ldots,X_m)[\oo']\\[1ex]
&=(1-2m)C_{m}( X_1,\ldots,X_m)[\oo]\\[1ex]
&\hspace{0,45cm} -2\sum_{j=1}^m\big((c_\infty+f_j) D_{m+1}( X_1,\ldots,X_m, X_j)[h_j'\oo] -D_{m+1}( X_1,\ldots,X_m, X_j)[h_jh_j'\oo]\\[1ex]
&\hspace{1.9cm} -(c_\infty+f_j)^2 C_{m+1}( X_1,\ldots,X_m, X_j)[ \oo] -C_{m+1}( X_1,\ldots,X_m, X_j)[h_j^2 \oo]\\[1ex]
&\hspace{1.9cm} +2(c_\infty+f_j) C_{m+1}( X_1,\ldots,X_m, X_j)[h_j \oo]\big).
\end{align*}
Combining the last three identities, Lemma~\ref{L:MP1a}, Lemma~\ref{L:MP2a}, and using a standard density argument  we get that \eqref{CM12b} holds for $E_m$.
The claim for the operator $E_m'$  follows similarly.
Finally, the Lipschitz continuity property is obtained from \eqref{CM12b} and~\eqref{LLC}-\eqref{LLC'}.
\end{proof}

Since our goal is to establish the smoothness of $\Phi$, cf. \eqref{RegP}, we next prove that operators~$E_m$ and~$E_m'$ with $E_m\in\{C_m,\, D_m\}$   depend  smoothly on the  variable $X\in\cO_r$.
This requires some additional notation.
Given $Y\coloneqq{}(u,v)\in H^r(\R)^2,$ we set
\begin{align}\label{notat2}
\ov\delta_{[ x, s]}Y\coloneqq{} u(x)-v(x-s)\quad\text{and}\quad \ov\delta'_{[ x, s]}Y\coloneqq{} v(x)-u(x-s),\qquad x,\, s\in\R.
\end{align}

Given $n,\, m,\, p\in\N$, $m\geq1$, $X_i\in \cO_r$, $1\leq i\leq m+p$,   $Y_i\in H^r(\R)^2$, $1\leq i\leq n,$  $\oo\in L_2(\R)$, and~${E\in\{C,\, D\}},$
 we set
\begin{equation*}
\begin{aligned}
E_{n,m,p}(X_1,\ldots,X_{m+p})[Y_1,\ldots,Y_n,\oo](x)&\coloneqq{}\int_\R\frac{s^j\oo(x-s)\big(\prod_{i=m+1}^{m+p}\delta_{[ x, s]}X_i\big)\prod_{i=1}^n\ov\delta_{[ x, s]}Y_i}{\prod_{i=1}^{m}\big[s^2+ (\delta_{[ x, s]}X_i)^2\big]}\, ds,\\[1ex]
 E'_{n,m,p}(X_1,\ldots,X_{m+p})[Y_1,\ldots,Y_n,\oo](x)&\coloneqq{}\int_\R\frac{s^j\oo(x-s)\big(\prod_{i=m+1}^{m+p}\delta'_{[ x, s]}X_i\big)\prod_{i=1}^n\ov\delta'_{[ x, s]}Y_i}{\prod_{i=1}^{m}\big[s^2+ (\delta'_{[ x, s]}X_i)^2\big]}\, ds
\end{aligned}
\end{equation*}
 for $x\in\R$, where $j=0$ if $E=C$ and $j=1$ for $E=D$. 
We point out that, given~${E\in\{C,\, D\}}$ and~$m\geq 1$, we have ${E_{0,m,0}(X)=E_m(X)}$ and ${E'_{0,m,0}(X)=E'_m(X)}$.
Hence  the latter formulas extend our previous notation introduced in~\eqref{CDms}.
 Setting $Y_i=(u_i,v_i)$, $1\leq i\leq n,$ it holds that
\begin{align*}
&\hspace{-0.5cm}E_{n,m,p}(X_1,\ldots,X_{m+p})[Y_1,\ldots,Y_n,\oo]\\[1ex]
&=\sum_{S\subset\{1,\ldots,n\}}(-1)^{|S^c|}\Big(\prod_{j\in S}u_j\Big)E_{0,m,p}(X_1,\ldots,X_{m+p})\Big[\oo\prod_{j\in S^c}v_j\Big],\\[1ex]
&\hspace{-0.5cm}E_{n,m,p}'(X_1,\ldots,X_{m+p})[Y_1,\ldots,Y_n,\oo]\\[1ex]
&=\sum_{S\subset\{1,\ldots,n\}}(-1)^{|S^c|}\Big(\prod_{j\in S}v_j\Big)E_{0,m,p}'(X_1,\ldots,X_{m+p})\Big[\oo\prod_{j\in S^c}u_j\Big],
\end{align*}
where for each $S\subset \{1,\ldots,n\}$ we set $S^c\coloneqq{}\{1,\ldots,n\}\setminus S$.

Moreover, letting $X_j\coloneqq{}(f_j,h_j)$, $m+1\leq j\leq m+p$, it holds that 
\begin{align*}
&\hspace{-0.5cm}E_{0,m,p}(X_1,\ldots,X_{m+p})[\oo]\\[1ex]
&=\sum_{S\subset\{m+1,\ldots,m+p\}}(-1)^{|S^c|}\Big(\prod_{j\in S}(c_\infty+f_j)\Big)E_{m}(X_1,\ldots,X_{m})\Big[\oo\prod_{j\in S^c}h_j\Big],\\[1ex]
&\hspace{-0.5cm}E_{0,m,p}'(X_1,\ldots,X_{m+p})[\oo]\\[1ex]
&=\sum_{S\subset\{m+1,\ldots,m+p\}}(-1)^{|S^c|}\Big(\prod_{j\in S}(h_j-c_\infty)\Big)E_{m}'(X_1,\ldots,X_{m})\Big[\oo\prod_{j\in S^c}f_j\Big].
\end{align*}

 Recalling Lemma~\ref{L:MP12b}, we deduce for $E\in\{C, \,C',\, D,\, D'\}$, that
 \begin{align}\label{estMUL}
 \|E_{n,m,p}(X_1,\ldots,X_{m+p})[Y_1,\ldots,Y_n,\cdot]\|_{\kL(L_2(\R), H^1(\R))}\leq C\prod_{i=1}^n\|Y_i\|_{H^r},
\end{align}  
with  $C$ is independent of $Y_i$, $1\leq i\leq n$.

Finally, given~$E\in\{C, \,C',\, D,\, D'\}$,  $n,\, m,\, p\in\N$, $m\geq1$, $Y_i\in H^r(\R)^2$, $1\leq i\leq n,$ and~${X\in\cO_r}$, we define
\begin{align}\label{Emnk}
E_{m,p}^n(X)[Y_1,\ldots,Y_n]\coloneqq{} E_{n,m,p}(X,\ldots,X)[Y_1,\ldots,Y_n,\cdot]\in\kL(L_2(\R), H^1(\R)).
\end{align} 
The estimate \eqref{estMUL} shows that  $E_{m,p}^n:\cO_r\to \kL^n_{\rm sym}(H^r(\R)^2,\kL(L_2(\R), H^1(\R)))$ (if $n=0$ we identify $\kL^n_{\rm sym}(H^r(\R)^2,\kL(L_2(\R), H^1(\R)))$ with  
$\kL(L_2(\R), H^1(\R))$).
In the next lemma we establish  the Fr\'echet differentiability  of $E_{m,p}^n.$
\begin{lemma}\label{L:MPsm}
Given  $n,\, m,\, p\in\N$, $m\geq1$, and $X\in\cO_r$, the   operator $E_{m,p}^n$ is Fr\'echet differentiable in~$X$ and its Fr\'echet derivative is given by 
\begin{align*}
\p E_{m,p}^n(X)[Y][Y_1,\ldots,Y_n]=pE_{m,p-1}^{n+1}(X)[Y_1,\ldots,Y_n, Y]-2m E_{m+1,p+1}^{n+1}(X)[Y_1,\ldots,Y_n, Y]
\end{align*}
for $Y,\,Y_1,\ldots, Y_n\in H^r(\R)^2$. Consequently, for each $n,\, m,\, p\in\N$, $m\geq1$, we have 
   $$E_{m,p}^n\in {\rm C}^\infty(\cO_r, \kL^n_{\rm sym}(H^r(\R)^2,\kL(L_2(\R), H^1(\R)))).$$
\end{lemma}
\begin{proof}
Setting 
\begin{align*}
R(X,Y)[Y_1,\ldots,Y_n]&\coloneqq{}E_{m,p}^n(X+Y)[Y_1,\ldots,Y_n]-E_{m,p}^n(X)[Y_1,\ldots,Y_n]\\[1ex]
&\hspace{0.55cm}-pE_{m,p-1}^{n+1}(X)[Y_1,\ldots,Y_n, Y]+2m E_{m+1,p+1}^{n+1}(X)[Y_1,\ldots,Y_n, Y],
\end{align*}
elementary algebraic manipulations lead us to the following identity
\begin{align*}
R(X,Y)[Y_1,\ldots,Y_n]&=\sum_{j=0}^{p-1}(p-j-1)R_{1,j}-\sum_{j=0}^{m-1}R_{2,j}+\sum_{j=0}^{m-1}\sum_{l=0}^{m-j-1}2R_{3,j,l},
\end{align*}
where
\begin{align*}
R_{1,j}&=E_{n+2,m,p-2}(\underset{m}{\underbrace{X+Y,\ldots,X+Y}},\underset{j}{\underbrace{X+Y,\ldots,X+Y}},\underset{p-2-j}{\underbrace{X,\ldots,X}})[Y_1,\ldots,Y_n,Y,Y],\\[1ex]
R_{2,j}&=(1+2p)E_{n+2,m+1,p}(\underset{m-j}{\underbrace{X+Y,\ldots,X+Y}},\underset{j+1}{\underbrace{X,\ldots,X}},\underset{p}{\underbrace{X,\ldots,X}})[Y_1,\ldots,Y_n,Y,Y]\\[1ex]
&\hspace{0,5cm}+pE_{n+3,m+1,p-1}(\underset{m-j}{\underbrace{X+Y,\ldots,X+Y}},\underset{j+1}{\underbrace{X,\ldots,X}},\underset{p-1}{\underbrace{X,\ldots,X}})[Y_1,\ldots,Y_n,Y,Y,Y],\\[1ex]
R_{3,j,l}&=2E_{n+2,m+2,p+2}(\underset{m-j-l}{\underbrace{X+Y,\ldots,X+Y}},\underset{j+l+2}{\underbrace{X,\ldots,X}},\underset{p+2}{\underbrace{X,\ldots,X}})[Y_1,\ldots,Y_n,Y,Y]\\[1ex]
&\hspace{0,5cm}+E_{n+3,m+2,p+1}(\underset{m-j-l}{\underbrace{X+Y,\ldots,X+Y}},\underset{j+l+2}{\underbrace{X,\ldots,X}},\underset{p+1}{\underbrace{X,\ldots,X}})[Y_1,\ldots,Y_n,Y,Y,Y].
\end{align*}
Hence, for all $Y$ sufficiently close to $X$ in $H^r(\R)^2$ it follows from Lemma~\ref{L:MP12b}, by arguing as in the derivation of \eqref{estMUL}, that 
\begin{align*}
\|R(X,Y)[Y_1,\ldots,Y_n]\|_{\kL(L_2(\R), H^1(\R))}\leq C\|Y\|^2_{H^r}\prod_{i=1}^n\|Y_i\|_{H^r},
\end{align*}
and the claim follows.
\end{proof}

We complete this section by establishing \eqref{RegP}.

\begin{cor}\label{Cor:1} It holds that $\Phi\in{\rm C}^{\infty}(\cO_r,H^{r-1}(\R)^2)$.
\end{cor}
\begin{proof}
The claim follows from \eqref{Phi1}-\eqref{Phi2}, Lemma~\ref{L:MP0}~(ii), and Lemma~\ref{L:MPsm}, by using the algebra property of $H^{r-1}(\R)$ and the embedding $H^1(\R)\hookrightarrow H^{r-1}(\R).$
\end{proof}

\section{The generator property and the proof of the main results}\label{Sec:3} 
In the first part of this section we show that the evolution problem \eqref{eq:F} is parabolic in $\cO_r$ by establishing the generator property \eqref{Gen}. In the second part we prove our main results.
With respect to the first task, let $X=(f,h)\in\cO_r$ be fixed. 
We can represent the Fr\'echet derivative~$\p\Phi(X)$ as the matrix operator
\begin{align*}
\p\Phi(X)
=\begin{pmatrix}
\p_f \Phi_1(X)&\p_h \Phi_1(X)\\[1ex]
\p_f \Phi_2(X)&\p_h \Phi_2(X)
\end{pmatrix}\in\kL(H^{r}(\R)^2,H^{r-1}(\R)^2).
\end{align*}
Though the coupling terms in \eqref{Phi1}-\eqref{Phi2} involve highest order derivatives of both unknowns, 
 Lemma~\ref{L:MPsm} and \eqref{estMUL} show that the  off-diagonal entry $\p_h \Phi_1(X)$  can be treated as being a perturbation.
Indeed, recalling Lemma~\ref{L:MPsm}, we obtain in particular for $E\in\{C,\, D\}$ that 
\begin{align}\label{F:F}
\p E_1(X)[Y]=-2E_{2,1}^1(X)[Y]=2E_{0,2,1}(X,X,X)[v\,\cdot]-2uE_{0,2,1}(X,X,X)
\end{align}
for all $Y=(u,v)\in H^r(\R)^2.$ Using this formula, it follows from~\eqref{Phi1} that
\begin{align*}
\p_h\Phi_1(X)[v]&=\frac{\Theta_2}{\pi}\Big((c_\infty+f)f'\big(C_1(X)[v']+2C_{0,2,1}(X,X,X)[vh']\big)\\[1ex]
&\hspace{1.25cm}-f'\big(C_1(X)[hv'+vh']+2C_{0,2,1}(X,X,X)[vhh']\big)\\[1ex]
&\hspace{1.25cm}+D_1(X)[v']+2D_{0,2,1}(X,X,X)[vh']\big)
\end{align*}
and \eqref{estMUL} yields 
\begin{align*}
\|\p_h\Phi_1(X)[v]\|_{H^{r-1}}\leq C\|v\|_{H^1}\qquad\text{for all $v\in H^r(\R).$}
\end{align*}
In view of \cite[Theorem I.1.6.1]{Am95} and of the property $\|v\|_{H^1}\leq \nu\|v\|_{H^r}+C(\nu)\|v\|_{H^{r-1}},$ $v\in~H^r(\R)$,  
which holds for any given arbitrary small $\nu>0$, we conclude that \eqref{Gen} is satisfied provided the diagonal entries are analytic generators, that is
\begin{align}\label{2Gen}
-\p_f\Phi_1(X)\in\mathcal{H}(H^{r}(\R),H^{r-1}(\R))\qquad\text{and}\qquad -\p_h\Phi_2(X)\in\mathcal{H}(H^{r}(\R),H^{r-1}(\R)).
\end{align} 

To establish the generator property for $\p_f\Phi_1(X)$ we follow the strategy from~\cite[Section~4]{AM21x} (see also \cite{E94, ES97} where similar arguments are used in other contexts).
To this end we first deduce from Lemma~\ref{L:MP0}~(ii) that the mapping
\[
[u\mapsto \bB(u)]:H^{r}(\R)\to \kL(H^{r-1}(\R))
\]  
is smooth. 
In view of this property, of \eqref{F:F}, and of \eqref{estMUL}  we infer from \eqref{Phi1} that
\begin{align*}
\p_f \Phi_1(X)[u]&=\Theta_1\big(\bB(f)[u']+\p\bB(f)[u][f']\big)+a(X)u'+T_{\rm lot}[u],
\end{align*}
where, according to Lemma~\ref{L:MP12b}, 
\begin{align*} 
a(X)\coloneqq{}\frac{\Theta_2}{\pi}\big((c_\infty+f)C_1(X)[h']-C_1(X)[hh']\big)\in H^1(\R).
\end{align*}
Moreover, $T_{\rm lot}$ is a lower order operator, more precisely
\begin{align}\label{L:LOT}
\|T_{\rm lot}[u]\|_{H^{r-1}}\leq C\|u\|_{H^1} \qquad\text{for all $u\in H^r(\R)$}.
\end{align}
We next consider the continuous path $[\tau\mapsto \Psi(\tau)]:[0,1]\to\kL(H^{r}(\R),H^{r-1}(\R))$ with
\begin{align*}
\Psi(\tau)[u]&\coloneqq{}\Theta_1\big(\bB(\tau f)[u']+\tau\p\bB(f)[u][f']\big)+\tau a(X)u'+\tau T_{\rm lot}[u].
\end{align*}
Observe that $\Psi(1)= \p_f \Phi_1(X)$ and $\Psi(0)$ is the Fourier multiplier
\[
\Psi(0)=\Theta_1\bB(0)\circ\frac{d}{dx}=\Theta_1 H\circ\frac{d}{dx}=\Theta_1\Big(-\frac{d^2}{dx^2}\Big)^{1/2}
\]
with symbol $[\xi\mapsto \Theta_1|\xi|].$
The  next step is to approximate $\Psi(\tau)$ locally by certain Fourier multipliers, see Lemma~\ref{L:Ap} below.
Therefore, we  choose  for each $\e\in(0,1)$, a finite ${\e}$-localization family, that is  a family  
\[\{\pi_j^\e\,:\, -N+1\leq j\leq N\}\subset  {\rm C}^\infty(\R,[0,1]),\]
with $N=N(\e)\in\N $ sufficiently large, such that
\begin{align*}
\bullet\,\,\,\, \,\,  & \text{$ \supp \pi_j^\e $ is an interval of length $\e$ for all $|j|\leq N-1$, $\supp \pi_{N}^\e\subset(-\infty,-1/\e]\cup [1/\e,\infty)$;} \\[1ex]
\bullet\,\,\,\, \,\, &\text{ $ \pi_j^\e\cdot  \pi_l^\e=0$ if $[|j-l|\geq2, \max\{|j|, |l|\}\leq N-1]$ or $[|l|\leq N-2, j=N];$} \\[1ex]
\bullet\,\,\,\, \,\, &\text{ $\sum_{j=-N+1}^N(\pi_j^\e)^2=1;$} \\[1ex]
 \bullet\,\,\,\, \,\, &\text{$\|(\pi_j^\e)^{(k)}\|_\infty\leq C\e^{-k}$ for all $ k\in\N, -N+1\leq j\leq N$.} 
\end{align*} 
To each finite $\e$-localization family we associate  a second family   
$$\{\chi_j^\e\,:\, -N+1\leq j\leq N\}\subset {\rm C}^\infty(\R,[0,1])$$ such that
\begin{align*}
\bullet\,\,\,\, \,\,  &\text{$\chi_j^\e=1$ on $\supp \pi_j^\e$, $\supp\chi_N^\e\subset [|x|\geq 1/\e-\e]$;} \\[1ex]
\bullet\,\,\,\, \,\,  &\text{$\supp \chi_j^\e$ is an interval  of length $3\e$ and with the same midpoint as $ \supp \pi_j^\e$, $|j|\leq N-1$.} 
\end{align*}

\begin{lemma}\label{L:Ap}
Let $X\in\cO_r$ be fixed and chose $r'\in(3/2,r)$.
Given   $\nu>0$,  there exist $\e\in(0,1)$, a finite $\e$-locali\-za\-tion family  $\{\pi_j^\e\,:\, -N+1\leq j\leq N\} $,  a positive constant $K=K(\e,X)$, 
and   bounded operators~$\bA_{j,\tau}\in\kL(H^r(\R), H^{r-1}(\R)), $ $j\in\{-N+1,\ldots,N\}$ and $\tau\in[0,1]$, such that 
 \begin{equation}\label{D1}
  \|\pi_j^\e \Psi(\tau) [u]-\bA_{j,\tau}[\pi^\e_j u]\|_{H^{r-1}}\leq \nu \|\pi_j^\e u\|_{H^r}+K\|  u\|_{H^{r'}}
 \end{equation}
 for all $ -N+1\leq j\leq N $, $\tau\in[0,1],$  and  $u\in H^r(\R)$. 
 The operators $\bA_{j,\tau}$ are defined  by 
  \begin{align*} 
 \bA_{j,\tau }\coloneqq{} \alpha_\tau(x_j^\e)\Big(-\frac{d^2}{dx^2}\Big)^{1/2}+\beta_\tau (x_j^\e)\frac{d}{dx}, \quad |j|\leq N-1, \qquad\bA_{N,\tau }\coloneqq{}    \Theta_1 \Big(-\frac{d^2}{dx^2}\Big)^{1/2},
 \end{align*}
 where  $x_j^\e\in \supp  \pi_j^\e,$ $|j|\leq N-1,$ and  with functions $\alpha_\tau,\, \beta_\tau$ given by
 \begin{align*}
 \alpha_\tau\coloneqq{}\frac{1+(1-\tau) f'^2}{1+f'^2}\Theta_1, \qquad  \beta_\tau\coloneqq{} \frac{\tau \Theta_1}{\pi}B_{1,1}^0(f)[f']+\tau a(X).   
 \end{align*} 
\end{lemma}
\begin{proof}
As shown in the proof of \cite[Theorem 4.3]{AM21x}  (in a more general context), if $\e$ is chosen sufficiently small, then for all  $\tau\in[0,1] $  and  $u\in H^r(\R) $ we have
 \begin{align*}
  & \Big\|\pi_j^\e \big(\bB(\tau f)[u']+\tau\p\bB(f)[u][f']\big)-\frac{\alpha_\tau(x_j^\e)}{\Theta_1} \Big(-\frac{d^2}{dx^2}\Big)^{1/2}[\pi^\e_j u]-\frac{\tau}{\pi}B_{1,1}^0(f)[f'](x_j^\e)(\pi_j^\e u)'\Big\|_{H^{r-1}}\\[1ex]
  &\hspace{3cm}\leq \frac{\nu}{2|\Theta_1|} \|\pi_j^\e u\|_{H^r}+K\|  u\|_{H^{r'}},\qquad |j|\leq N-1,
 \end{align*}
 and 
  \begin{align*}
\Big\|\pi_N^\e \big(\bB(\tau f)[u']+\tau\p\bB(f)[u][f']\big)- \Big(-\frac{d^2}{dx^2}\Big)^{1/2}[\pi^\e_N u]\Big\|_{H^{r-1}}
  \leq \frac{\nu}{2|\Theta_1|} \|\pi_j^\e u\|_{H^r}+K\|  u\|_{H^{r'}}.
 \end{align*}

We next recall, see e.g.  \cite[Eq. 2.1]{AM21x},  there exists a constant $C>0$ such that 
\begin{align*}
\|ab\|_{H^{r-1}}\leq C(\|a\|_{H^{r-1}}\|b\|_\infty+\|a\|_\infty\|b\|_{H^{r-1}}) \qquad\text{for all $a,$ $ b\in H^{r-1}(\R)$.}
\end{align*} 
Using this estimate together with the identity $\chi_j^\e\pi_j^\e=\pi_j^\e$, $-N+1\leq j\leq N$, we get in view of the relation $a(X)\in {\rm C}^{1/2}(\R)$   that 
\begin{align*}
\|\pi_j^\e a(X)u'-a(X)(x_j^\e)(\pi^\e_j u)'\|_{H^{r-1}}&\leq\|\chi_j^\e \big(a(X)-a(X)(x_j^\e)\big)(\pi^\e_j u)'\|_{H^{r-1}}+K\|u\|_{H^{r-1}}\\[1ex]
&\leq C\|\chi_j^\e \big(a(X)-a(X)(x_j^\e)\big)\|_\infty\|\pi^\e_j u\|_{H^{r}}+K\|u\|_{H^{r'}}\\[1ex]
&\leq  \frac{\nu}{4} \|\pi_j^\e u\|_{H^r}+K\|  u\|_{H^{r'}},\qquad |j|\leq N-1,
\end{align*}
if $\e$ is sufficiently small,
respectively, in view of the fact that $a(X)$ vanishes at infinity,
\begin{align*}
\|\pi_N^\e a(X)u' \|_{H^{r-1}}&\leq\|\chi_N^\e a(X)(\pi_N^\e u)'\|_{H^{r-1}}+K\|u\|_{H^{r-1}}\\[1ex]
&\leq C\|\chi_N^\e a(X)\|_{H^{r-1}}\|\pi^\e_N u\|_{H^{r}}+K\|u\|_{H^{r'}}\\[1ex]
&\leq  \frac{\nu}{4} \|\pi_N^\e u\|_{H^r}+K\|  u\|_{H^{r'}}
\end{align*}
for all  $u\in H^r(\R).$
These estimates together with \eqref{L:LOT} lead us to \eqref{D1}.
\end{proof}

Let us observe there exists $\eta\in(0,1)$ such that the symbols of the Fourier multipliers identified in Lemma~\ref{L:Ap} satisfy
\begin{align*}
\eta\leq -\alpha_\tau\leq \frac{1}{\eta}\quad\text{and}\quad \|\beta_\tau||_\infty\leq \frac{1}{\eta}\qquad\text{for all $\tau\in[0,1]$.}
\end{align*}
Classical Fourier analysis arguments then show there exists $\kappa_0=\kappa_0(\eta)\geq 1$ such that 
\begin{align}
\bullet &\quad \mbox{$\lambda-\bA_{\alpha,\beta}\in \kL(H^r(\R),H^{r-1}(\R))$ is an isomoprphism for all $ \re\lambda\geq 1,$}\label{L:FM1}\\[1ex]
\bullet &\quad  \kappa_0\|(\lambda-\bA_{\alpha,\beta})[u]\|_{H^{r-1}}\geq |\lambda|\cdot\|u\|_{H^{r-1}}+\|u\|_{H^r}, \qquad \forall\, u\in H^r(\R),\, \re\lambda\geq 1\label{L:FM2},
\end{align}
uniformly for $\bA_{\alpha,\beta}\coloneqq{} \alpha(-d^2/dx^2)^{1/2}+\beta (d/dx)$, where $-\alpha\in[\eta,1/\eta],$ $|\beta|\leq 1/\eta$.
The properties~\eqref{L:FM1}-\eqref{L:FM2} combined with Lemma~\ref{L:Ap} enable us to obtain the desired generator property for~${\p_f\Phi_1(X)}$.

\begin{prop}\label{P:pfo1}
Given $X\in\cO_r$, it holds that $ -\p\Phi(X)\in\mathcal{H}(H^{r}(\R)^2,H^{r-1}(\R)^2).$
\end{prop}
\begin{proof}
According to our discussion above it remains to establish \eqref{2Gen}. To prove  the generator property for~$\p_f\Phi_1(X)$,
we may argue as in \cite[Theorem~4.1]{AM21x} to  find, in view of \eqref{L:FM1}-\eqref{L:FM2} and of Lemma~\ref{L:Ap}, constants  $\kappa=\kappa(X)\geq1$  and $\omega=\omega(X)>0 $ such that 
  \begin{align}\label{KDED}
   \kappa\|(\lambda-\Psi(\tau  ))[u]\|_{H^{r-1}}\geq |\lambda|\cdot\|u\|_{H^{r-1}}+ \|u\|_{H^{r}}
 \end{align}
for all   $\tau\in[0,1],$   $\re \lambda\geq \omega$, and  $u\in H^{r}(\R)$.
Choosing $\omega\geq 1$, it follows from  \eqref{L:FM1}, in view of~$\Psi(0)=\bA_{\Theta_1,0}$, that $\omega-\Psi(0)\in \kL(H^{r}(\R), H^{r-1}(\R))$ is an isomorphism.
The method of continuity, cf. \cite[Proposition I.1.1.1]{Am95}, and \eqref{KDED} then imply  that~${\omega-\p_f\Phi_1(X)\in \kL(H^{r}(\R), H^{r-1}(\R))}$ is an isomorphism too. 
From this property and \eqref{KDED} (with $\tau=1$) we deduce that indeed~$-\p_f\Phi_1(X)$ belongs to~$\mathcal{H}(H^{r}(\R),H^{r-1}(\R))$.  
Since the generator property for $\p_h\Phi_2(X)$ follows by using similar arguments (which we therefore omit), this proves our claim.  
\end{proof}

We are now in a position to prove Theorem~\ref{MT1}.
\begin{proof}[Proof of Theorem~\ref{MT1}]
The properties \eqref{RegP} and \eqref{Gen} enable us to use the abstract parabolic theory from  \cite[Chapter 8]{L95} 
in the context of the evolution problem \eqref{eq:F}.
More precisely, given~$X_0\in \cO$,   \cite[Theorem~8.1.1]{L95} implies there exists a time $T>0$ and a solution~${X=X(\cdot;X_0)}$ to~\eqref{eq:F} 
such that\footnote{Given $\alpha\in(0,1)$, $T>0$,  and a Banach space $X$,   let $B((0,T], X)$   denote the Banach space of all bounded functions from $(0,T]$  into $X$. 
The Banach space ${\rm C}^\alpha_\alpha((0,T], X)$ is then defined as
 \[
 {\rm C}^\alpha_\alpha((0,T], X)\coloneqq{}\Big\{f\in B((0,T], X)\,:\, \|f\|_{C^\alpha_\alpha}\coloneqq{}\|f\|_\infty+\sup_{s\neq t}\frac{\|t^\alpha f(t)-s^\alpha f(s)\|_X}{|t-s|^\alpha}<\infty\Big\}.
 \]
 }
\[ X\in {\rm C}([0,T],\cO_r)\cap {\rm C}^1([0,T], H^{r-1}(\R)^2)\cap  {\rm C}^{\alpha}_{\alpha}((0,T], H^r(\R)^2)\] 
for some  $\alpha\in(0,1)$ (actually, since \eqref{eq:F} is autonomous, for all $\alpha\in(0,1)$).
Moreover, the solution is unique within the class 
\[
  \bigcup_{\alpha\in(0,1)}{\rm C}^{\alpha}_{\alpha}((0,T], H^r(\R)^2) \cap {\rm C}([0,T],\cO_r)\cap {\rm C}^1([0,T], H^{r-1}(\R)^2).
 \]
To prove that  the solution is unique  in ${\rm C}([0,T],\cO_r)\cap {\rm C}^1([0,T], H^{r-1}(\R)^2),$  
 we assume by contradiction  there exist two solutions $X_i:[0,T]\to\cO_r$, $i=1,\, 2$, to~\eqref{eq:F} such that~$X_1(0)=X_2(0)$ and~$X_1(t)\neq  X_2(t)$ for all $t\in(0,T]$. 
 Let $r'\in (3/2,r)$ be arbitrary  and set~${\alpha\coloneqq{} r-r'\in(0,1)}$. 
 The mean value theorem together with the inequality $ \|a\|_{H^{r'}}\leq \|a\|_{H^{r-1}}^\alpha \|a\|_{H^{r}}^{1-\alpha},$ $a\in H^r(\R),$
imply there exists a constant $C>0$ such that
 \begin{equation}\label{AdReg}
\|X_i(t)-X_i(s)\|_{H^{r'}}\leq C|t_1-t_2|^\alpha,\qquad s,\, t\in[0, T], \, i=1,\, 2.
 \end{equation}
Hence $X_i\in {\rm C}^{\alpha}([0,T],H^{r'}(\R)^2))\hookrightarrow {\rm C}^{\alpha}_{\alpha}((0,T], H^{r'}(\R)^2)$, and 
 \cite[Theorem 8.1.1]{L95} applied in the context of  \eqref{eq:F} with $r$ replaced by $r'$
ensures that $X_1=X_2$ in $[0,T],$ which contradicts our assumption.
This  unique local solution can be extended up to a maximal existence time~${T^+=T^+(X_0)}$, see \cite[Section 8.2]{L95}.\medskip

  The continuous dependence of the solution on the initial data stated at (i) follows from~\cite[Proposition~8.2.3]{L95}. \medskip
  
The proof of (ii)  uses a parameter trick which was successfully applied also to other problems, cf., e.g., \cite{An90, ES96, PSS15, MBV19, AM21x}.
Since the details are very similar to those in \cite[Theorem~1.2~(ii)]{AM21x}  we omit them.\medskip

To prove (iii) we assume  there exists a maximal solution   
$X=X(\cdot;X_0)$ to \eqref{eq:F} with~${T^+\!<\infty}$ and such that  
\[
 \sup_{t\in [0,T^+)}\|X(t)\|_{H^r}<\infty\qquad\text{and}\qquad \liminf_{t\to T_+}{\rm dist\,}(\G_f^{c_\infty}(t),\G_h(t))= c_0>0. 
\]
Arguing as above, we deduce for some fixed $r'\in(3/2,r)$,   that  $X:[0,T^+)\to \cO_{r'}$ is H\"older  continuous. 
Applying \cite[Theorem 8.1.1]{L95} to~\eqref{eq:F} (with $r$ replaced by $r'$)  we may extend the solution~$X$   to  an  interval~$[0,T')$ with $T^+<T'$  and such that 
 $X\in {\rm C}([0,T'),\cO_{r'})\cap {\rm C}^1([0,T'), H^{r'-1}(\mathbb{R})^2).$
Moreover,  the parabolic smoothing property established at (ii) (with $r$ replaced by $r'$)  implies  that~$X\in {\rm C}^1((0,T'), H^{r}(\mathbb{R})^2)$, and this contradicts  the maximality property
 of $X$. This completes  the proof.
\end{proof}

We conclude this section with the proof of Proposition~\ref{MP1}.
\begin{proof}[Proof of Proposition~\ref{MP1}]
Since $\|X(t)\|_{H^r}\leq M$ for all $t\in[0,T^+)$, \eqref{eq:S4} and Lemma~\ref{L:Bdd}  imply there exists $C>0$ such that
\[
\Big\|\frac{dX(t)}{dt}\Big\|_\infty\leq C(1+M^4),\qquad t\in[0,T^+).
\] 
Therefore  $X^2:=(f^2,h^2)\in{\rm C}^1([0,T^+), L_2(\R)^2)$ has a bounded derivative.
Moreover, we also have that~${X^2:[0,T)\to H^r(\R)^2}$ is  bounded. 
Since ${\|a\|_{H^{r'}}\leq \|a\|_{2}^{1-r'/r} \|a\|_{H^{r}}^{r'/r},}$~${a\in H^r(\R),}$ the mean value theorem yields
 $X^2\in {\rm BUC}^{1-r'/r}([0,T^+), \cO_{r'})$, where $r'\in(3/2,r)$ is  fixed. 
 Hence, there exists $X_*\in H^{r'}(\R)^2$ such that $X^2(t)\to X_*=(f_*,h_*)$ in $H^{r'}(\R)^2$ for $t\to T^+$.
 
 Since  
\[
\liminf_{t\to T^+}{\rm dist\,}(\G_f^{c_\infty}(t),\G_h(t))=0,
\]
cf.  Theorem~\ref{MT1}~(iii), there exists sequences $t_n\nearrow T^+$ and $(x_n)\subset\R$ with 
\begin{align}\label{convo}
(c_\infty+f)(t_n,x_n)-h(t_n,x_n)\to0\qquad\mbox{for $n\to\infty$}.
\end{align}
We next show that $(x_n)$ is bounded. 
To this end we infer from the convergence  ${X^2(t)\to X_*}$ in~${H^{r'}(\R)^2}$  there exists $n_0\in\N$ such that 
$|f(t_n,x)|+|h(t_n,x)|<c_\infty/2$ for all $n\geq n_0$ and~${|x|\geq n_0}$.
The latter inequality together with \eqref{convo} implies that $(x_n)$ is indeed bounded.

We may thus assume (after eventually subtracting a subsequence), that $x_n\to x_0$ in~$\R$. 
Since~$X_*$ is continuous and $X^2(t_n)\to X_*$ in $H^{r'}(\R)^2$ we get~${X(t_n,x_n)\to(\sqrt{f_*}(x_0), \sqrt{h_*}(x_0))}$.
 The relation \eqref{convo} now yields~${c_\infty+\sqrt{f_*}(x_0)=\sqrt{h_*}(x_0)}.$ 
 Finally, since  $X^2(t_n,x_0)\to X_*(x_0)$, together with the latter identity we conclude that~${c_\infty+f(t_n,x_0)-h(t_n,x_0)\to0,}$ and  therefore
 $$\liminf_{t\to T^+}\,(c_\infty+f(t,x_0)-h(t,x_0))=0.$$
 
 In order to prove the second claim we argue by contradiction and assume there exists~$x_0\in\R$ and~$\delta>0$ such that 
 \[
 \liminf_{t\to T^+}\sup_{\{|x-x_0|\leq \delta\}}(c_\infty+f(t,x)-h(t,x))=0. 
\]
Since  \eqref{eq:F} is invariant under horizontal translations we may  assume without loss of generality that $x_0=0.$
Hence, there exists a sequence $(t_n)$ with $t_n\nearrow T^+$ and  
\begin{equation}\label{bdd234}
c_\infty+f(t_n)-h(t_n)\to0\qquad\mbox{in $L_\infty([-\delta,\delta]).$}
\end{equation}
Recalling Lemma~\ref{L:Bdd}, we find a constant $c_1=c_1(M)$ such that the velocity~${v_2(t)\!=\!(v_2^1(t),v_2^2(t))}$ satisfies
\begin{align}\label{bdd2}
\|v_2(t)\|_{L_\infty(\0_2(t))}\leq c_1,\qquad t\in[0,T^+).
\end{align}
Given $t\in(T^+-\delta/c_1, T^+)$, let $R(t)\coloneqq{}\delta+c_1(t-T^+)$.
Then $R$ is a positive  increasing function with $R(t)\to\delta$ for $t\to T^+$.
We further define the surface area 
\[
S(t)\coloneqq{}\int_{-R(t)}^{R(t)} (c_\infty+f(t,x)-h(t,x))\, dx,\qquad t\in(T^+-\delta/c_1, T^+).
\]
Let $n_0\in\N$ be fixed such that $t_n>T^+-\delta/c_1$ for all~${n\geq n_0}$.  
On the one hand, ${S(t_n)>0}$ for all~${n\geq n_0}$. 
Moreover, the dominated convergence theorem together with~\eqref{bdd234} immediately implies that~$S(t_n)\to 0$ for $n\to\infty$.

On the other hand, given $t\in(T^+-\delta/c_1, T^+)$, Stokes' theorem together with~${{\rm div\,} v_2(t)=0}$ in~${\0_2(t)}$ yields 
\begin{align*}
S'(t)=\int_{h(t,-R(t))}^{c_\infty+f(t,-R(t))}\big(c_1+v_2^1(t,-R(t),y)\big)\, dy+\int_{h(t,R(t))}^{c_\infty+f(t,R(t))}\big(c_1-v_2^1(t,-R(t),y)\big)\, dy.
\end{align*}
Hence, in view of the bound \eqref{bdd2}, we have $S'(t)\geq0$ for all $t\in(T^+-\delta/c_1, T^+)$. 
Consequently,~${S(t_n)\geq S(t_{n_0})>0}$ for all $n\geq n_0$, in contradiction to ~$S(t_n)\to 0$ for $n\to\infty$.
Hence, our assumption was false and the argument is complete.
\end{proof}

\appendix
 \section{An extension of Privalov's theorem}\label{Sec:A}
In this section we fix $p\in(1,\infty)$, $\alpha\in(0,1)$, $\oo\in {\rm BUC}^{\alpha}(\R)\cap L_p(\R)$, and a differentiable function $f:\R\to\R$ with $f'\in {\rm BUC}^{\alpha}(\R)$.
 We study  the function $v\coloneqq{}v(f)[\oo]:\R^2\setminus\Gamma_f\to\R^2$ given by the formula
\begin{align}\label{F:V}
v(x,y)\coloneqq{}v(f)[\oo](x,y)\coloneqq{}\frac{1}{2\pi}\int_\R\frac{(f(s)-y,x-s)}{(x-s)^2+(y-f(s))^2}\oo(s)\, ds,
\end{align}
 where 
 \[
\Gamma_f\coloneqq{}\{(x,f(x))\,:\, x\in\R\}. 
 \]
Let us first note that $v$ is the complex conjugate of a holomorphic function, see \eqref{PHIcont},  so that~$v$ is smooth, that is~${v=:(v^1,v^2)\in {\rm C}^\infty(\R^2\setminus\Gamma_f).}$
For this function we establish below several additional properties. 
In particular we extend Privalov's theorem, cf. e.g. \cite{JKL93}, and prove that~$v$ is $\alpha$-Hölder continuous 
 in the domains above and below the graph~$\Gamma_f$, cf. Theorem~\ref{T:Priva}.

As a first step we show in Lemma~\ref{L:Plemelj} that  the one-sided limits of $v$ when approaching a point on~$\Gamma_f$ from below or from above exist.
This is a consequence of the classical Plemelj formula,  cf. e.g. \cite{JKL93}, and of the observation that 
\begin{align}\label{F:V2}
v(x,y)=\overline{\frac{1}{2\pi i}\int_{\Gamma_f}\frac{\varphi(\xi)}{\xi-z}\, d\xi},\qquad  z=(x,y)\in \R^2\setminus\Gamma_f,
\end{align}
where the function $\varphi:\Gamma_f\to\C$ in the contour integral \eqref{F:V2} is defined by
\[
\varphi(\xi)=-\frac{\oo(1,-f')}{1+f'^2}(s),\qquad \xi=(s,f(s))\in\Gamma_f.
\]
We note that $\varphi\in {\rm BUC}^{\alpha}(\Gamma_f)$, that is $\varphi$ is bounded and
\[
[\varphi]_{\alpha}\coloneqq{} \sup_{\zeta\neq \xi\in \Gamma_f}\frac{| \varphi(\zeta)-\varphi(\xi)|}{|\zeta-\xi|^{\alpha}}<\infty.
\] 
It is suitable to introduce the function $F:\C\setminus\Gamma_f\to \C$ defined by 
\begin{align}\label{PHIcont}
F(z)\coloneqq{}\overline{v(z)}=\frac{1}{2\pi i}\int_{\Gamma_f}\frac{\varphi(\xi)}{\xi-z}\, d\xi.
\end{align}
It is not difficult to prove that this function is holomorphic.
\begin{lemma}\label{L:Plemelj}
Let
\[
\Omega_\pm\coloneqq{}\{(x,y)\in\R^2\,:\, \pm (y-f(x))>0\}.
\]
The restrictions $v_\pm\coloneqq{}v|_{\0_\pm}:\0_\pm\to\R^2$ of the function $v$  defined in \eqref{F:V} extend continuously up to $\Gamma_f$ and, given $x\in\R$, we have
\begin{align}\label{Vary}
v_\pm(x,f(x))=\frac{1}{2\pi }\PV\int_\R\frac{(f(s)-f(x),x-s)}{(x-s)^2+(f(x)-f(s))^2}\oo(s)\, ds\mp\frac{1}{2}\frac{\oo(1,f')}{1+f'^2}(x).
\end{align} 
\end{lemma}
\begin{proof}
Let  $z_0\coloneqq{}(x_0,f(x_0))\in\Gamma_f$.
In order to prove that $v_+$ can be extended continuously in~$z_0$ we consider  the polygonal path  $\Gamma_1\subset\overline{\0_+}$
defined by  the  segments $[{(x_0+1,f(x_0+1))}, (x_0+1,D)]$, $[(x_0+1,D),(x_0-1, D)]$, and $[(x_0-1,D),(x_0-1, f(x_0-1))]$  and oriented counterclockwise.  
 Here we set $D\coloneqq 1+2\|f'\|_\infty+ \max\{f(x_0-1),f(x_0+1)\}$.
 Moreover, we let
 \[
\Gamma_0\coloneqq{}\{(x,f(x))\,:\, |x-x_0|\leq 1\} 
 \]
 and we define the closed curve $\Gamma\coloneqq{}\Gamma_0+\Gamma_1$ which is again oriented counterclockwise.
 Additionally, we define the function $\wt\varphi\in{\rm BUC}^{\alpha}(\Gamma)$ by setting
 \[
\wt\varphi(\xi)\coloneqq{}
\left\{
\begin{array}{cllll}
\varphi(\xi)&,&\xi\in\Gamma_0,\\[1ex]
\varphi_+&,&\xi=(x_0+1,y), \, f(x_0+1)\leq y\leq D,\\[1ex]
\cfrac{(1+x_0-x)\varphi_-+(1+x-x_0)\varphi_+}{2}&,& \xi=(x,D),\, |x-x_0|\leq 1,\\[1ex]
\varphi_-&,&\xi=(x_0-1,y), \, f(x_0-1)\leq y\leq D,
\end{array}
\right. 
 \]
 where  $\varphi_\pm\coloneqq{} \varphi(x_0\pm 1,f(x_0\pm 1))$.
 It is not difficult to check that
 \begin{align}\label{UBound}
 \|\wt\varphi\|_\infty\leq \|\varphi\|_\infty,\quad [\wt\varphi]_{\alpha}\leq 2\|\varphi\|_\infty+ [\varphi]_{\alpha},\quad |\Gamma|\leq 7(\|f'\|_\infty+1).
 \end{align}
Given $z\in\0_+$  which is sufficiently close to $z_0$, it then holds
 \begin{align*}
 \int_{\Gamma_f}\frac{\varphi(\xi)}{\xi-z}\, d\xi&=\int_{\Gamma_f-\Gamma_0}\frac{\varphi(\xi)}{\xi-z}\, d\xi+\int_{\Gamma}\frac{\wt\varphi(\xi)}{\xi-z}\, d\xi-\int_{\Gamma_1}\frac{\wt\varphi(\xi)}{\xi-z}\, d\xi.
 \end{align*}
Since $z_0\in\Gamma_0$, Lebegue's dominated convergence shows that 
 \begin{align*}
 \int_{\Gamma_f-\Gamma_0}\frac{\varphi(\xi)}{\xi-z}\, d\xi-\int_{\Gamma_1}\frac{\wt\varphi(\xi)}{\xi-z}\, d\xi
 \underset{z\to z_0}\longrightarrow\int_{\Gamma_f-\Gamma_0}\frac{\varphi(\xi)}{\xi-z_0}\, d\xi-\int_{\Gamma_1}\frac{\wt\varphi(\xi)}{\xi-z_0}\, d\xi.
 \end{align*}
 Additionally, according to the Plemelj formula,   cf. e.g. \cite{JKL93}, it holds that 
  \begin{align*}
\frac{1}{2\pi i} \int_{\Gamma}\frac{\wt\varphi(\xi)}{\xi-z}\, d\xi\underset{z\to z_0}\longrightarrow\frac{1}{2\pi i}\PV\int_{\Gamma}\frac{\wt\varphi(\xi)}{\xi-z_0}\, d\xi+\frac{1}{2}\varphi(z_0).
 \end{align*}
 These two convergences imply that $v_+$ can indeed be extended continuously in~$z_0$, the value of the extension in $z_0$ being as given in formula~\eqref{Vary}.
 Finally, the corresponding claim for~$v_-$ follows  by   arguing similarly.
\end{proof}

We next prove in Lemma~\ref{L:Bdd} that $v$ is bounded in $\R^2\setminus\Gamma_f$. 
In fact we bound in Lemma~\ref{L:Bdd}  the~$L_\infty$-norm of~$v$ by a constant that depends explicitly on the norms of the functions $f$ and~$\oo$.

\begin{lemma}\label{L:Bdd}
There exists a constant $C$, which is independent of $f$ and $\oo$, such that  
\begin{align}\label{inftybdd}
\|v\|_\infty\leq  C(\|\oo\|_{p}+\|\oo\|_{{\rm BUC}^{\alpha}})(1+\|f'\|_{{\rm BUC}^{\alpha}})^2.
\end{align} 
\end{lemma}
\begin{proof}
We devise the proof in several steps.\medskip

\noindent{Step 1.}   In this step we provide bounds for  the restrictions of $v_{\pm}$ to $\Gamma_f$.
Given $x\in\R$, it follows from~\eqref{Vary} and Hölder's inequality that
\begin{align*}
|v_\pm(x,f(x))|&\leq \|\oo\|_\infty+\int_{\{|s|>1\}}\Big|\frac{\oo(x-s)}{s}\Big|\, ds+\Big|\PV\int_{-1}^1\frac{(f(x-s)-f(x),s)}{s^2+(f(x)-f(x-s))^2}\oo(x-s)\, ds\Big|\\[1ex]
&\leq \|\oo\|_\infty+C\|\oo\|_p+I_1+I_2,
\end{align*}
where
\begin{align*}
I_1&\coloneqq{}\Big|\int_{-1}^1\frac{(f(x-s)-f(x),s)}{s^2+(f(x-s)-f(s))^2}(\oo(x-s)-\oo(x)\, ds\Big|\leq[\oo]_{\alpha} \int_{-1}^1|s|^{\alpha-1}\, ds\leq C[\oo]_{\alpha},\\[1ex]
I_2&\coloneqq{}\|\oo\|_\infty\Big|\PV\int_{-1}^1\frac{(f(x-s)-f(x),s)}{s^2+(f(x)-f(x-s))^2}\, ds\Big|.
\end{align*}
Concerning $I_2$, we have
\begin{align*}
I_2&\leq \|\oo\|_\infty\int_0^1\Big|\frac{(f(x-s)-f(x),s)}{s^2+(f(x)-f(x-s))^2}+\frac{(f(x+s)-f(x),-s)}{s^2+(f(x)-f(x+s))^2}\Big|\, ds\\[1ex]
&\leq3\|\oo\|_\infty\int_0^1\frac{|f(x+s)-2f(x)-f(x-s)|}{s^2}\, ds\leq6\|\oo\|_\infty[f']_{\alpha}\int_0^1|s|^{\alpha-1}\, ds \leq C\|\oo\|_\infty[f']_{\alpha}.
\end{align*}
Gathering these estimates we conclude that 
\begin{align}\label{Bound1}
\|v_\pm|_{\Gamma_f}\|_{\infty}\leq C(\|\oo\|_{p}+\|\oo\|_{{\rm BUC}^{\alpha}})(1+\|f'\|_{{\rm BUC}^{\alpha}}).
\end{align}

\noindent{Step 2.}  Given $z=(x,y)\in\R^2,$ we set  $d(z)\coloneqq{}{\rm dist\,}(z,\Gamma_f)$. We next prove that
\begin{align}\label{Bound2}
\sup_{\{1/4\leq d(z)\}}|v_\pm(z)|\leq C\|\oo\|_{p}.
\end{align}
Indeed,  if $1/4\leq d(z)$, then $\sqrt{s^2+(y-f(x-s))^2}\geq \max\{1/4,\,|s|\}$ for all $s\in\R$ and together with Hölder's inequality  we conclude from \eqref{F:V} that
\begin{align*}
|v(z)|\leq\int_\R\frac{1}{\max\{1/4,\,|s|\}}|\oo(x-s)|\, ds\leq C\|\oo\|_p.
\end{align*}

\noindent{Step 3.} In this final step we prove that
\begin{align}\label{Bound3}
\sup_{\{0<d(z)<1/4\}}|v_\pm(z)|\leq C(\|\oo\|_{p}+\|\oo\|_{{\rm BUC}^{\alpha}})(1+\|f'\|_{{\rm BUC}^{\alpha}})^2.
\end{align}
We first consider the case when  $z\in\0_+$.
 We associate to $z$ a point  $z_\G=(x_0,f(x_0))\in\Gamma_f$ such that
\[
d(z)=|z-z_\G|\in(0,1/4).
\]
Let $\G=\Gamma_0+\Gamma_1$ and $\wt\varphi $ be as defined in the proof of Lemma~\ref{L:Plemelj} (with $z_\Gamma$ instead of $z_0$).
Recalling~\eqref{Bound1}, we have
\begin{align*}
|v_+(z)|\leq |v_+(z)-v_+(z_\G)|+|v_+(z_\G)|&\leq T_1+T_2+T_3+C(\|\oo\|_{p}+\|\oo\|_{{\rm BUC}^{\alpha}})(1+\|f'\|_{{\rm BUC}^{\alpha}}),
\end{align*}
where
\begin{align*}
T_1&\coloneqq{}\Big|\int_{\Gamma_f-\Gamma_0}\Big(\frac{\varphi(\xi)}{\xi-z}-\frac{\varphi(\xi)}{\xi-z_\G}\Big)\, d\xi\Big|,\qquad
T_2\coloneqq{}\Big|\int_{\Gamma_1}\Big(\frac{\wt\varphi(\xi)}{\xi-z}-\frac{\wt\varphi(\xi)}{\xi-z_\G}\Big)\, d\xi\Big|,\\[1ex]
T_3&\coloneqq{}\Big|\frac{1}{2\pi i}\int_{\Gamma}\frac{\wt\varphi(\xi)}{\xi-z}\, d\xi-\frac{1}{2\pi i}\PV\int_{\Gamma}\frac{\wt\varphi(\xi)}{\xi-z_\G}\, d\xi-\frac{1}{2}\varphi(z_\G)\Big|.
\end{align*}
Given $\xi\in\Gamma_1$, we have  $\min\{|\xi-z|,\,|\xi-z_\G|\}\geq |\xi-z_\G|-1/4\geq 3/4$ and~\eqref{UBound} yields 
\[
T_2\leq 2\|\varphi\|_\infty|\G_1|\cdot|z-z_\G|\leq C\|\oo\|_\infty(1+\|f'\|_{\infty}).
\]  
Moreover, since $\min\{|\xi-z|,\,|\xi-z_\G|\}\geq \max\{3/4,\, |s-x_0|/2\}$ for all~${\xi=(s,f(s))\in\Gamma_f-\Gamma_0}$, Hölder's inequality leads us to
\begin{align*}
T_1&\leq \frac{8}{3}\|\oo\|_p |z-z_\G| \Big(\int_\R \frac{1}{\max\{3/2, |s|\}^{p'}}\, ds\Big)^{1/p'}\leq C\|\oo\|_p,
\end{align*}
where $p'\in(1,\infty)$ is the adjoint exponent to $p$, that is $p^{-1}+{p'}^{-1}=1$. In order to estimate $T_3$ we first note that
\begin{align}\label{INTF}
\frac{1}{2\pi i}\int_{\Gamma}\frac{1}{\xi-z}\, d\xi=1 \qquad\text{and}\qquad \frac{1}{2\pi i}\PV\int_{\Gamma}\frac{1}{\xi-z_\G}\, d\xi=\frac{1}{2}.
\end{align}
The first relation follows from Cauchy's integral formula. 
The second identity  is a direct consequence of Plemelj's formula, cf. e.g. \cite{JKL93}. 
Using these two identities we get
\begin{align*}
T_3&=\Big|\frac{1}{2\pi i}\int_{\Gamma}\frac{\wt\varphi(\xi)-\wt \varphi(z_\G)}{\xi-z}\, d\xi-\frac{1}{2\pi i}\int_{\Gamma}\frac{\wt\varphi(\xi)-\wt\varphi(z_\G)}{\xi-z_\G}\, d\xi\Big|
\leq  T_{3a}+T_{3b},
\end{align*}
where
\begin{align*}
T_{3a}\coloneqq{}\Big| \int_{\Gamma_1}\frac{(z-z_\G)(\wt\varphi(\xi)-\wt \varphi(z_\G))}{(\xi-z)(\xi-z_\G)}\, d\xi\Big|\quad\text{and}\quad
T_{3b}\coloneqq{}\Big| \int_{\Gamma_0}\frac{(z-z_\G)(\varphi(\xi)- \varphi(z_\G))}{(\xi-z)(\xi-z_\G)}\, d\xi\Big|.
\end{align*}
The arguments used to estimate $T_2$   lead us to 
\begin{align*}
T_{3a}\leq\|\varphi\|_\infty|\G_1|\leq C\|\oo\|_\infty(1+\|f'\|_{\infty}).
\end{align*}
In order to estimate $T_{3b}$  we note that $|\xi-z|\geq d(z)=|z-z_\G|$ and  $|\xi-z_\G|\geq |s-x_0|$ for all~${\xi=(s,f(s))\in\Gamma_0}$, and therefore 
 \begin{align*}
T_{3b}&\leq (1+\|f'\|_\infty)[\varphi]_{\alpha}\int_{\{|s|<1\}}s^{\alpha-1}\, ds\leq C(1+\|f'\|_\infty)[\varphi]_{\alpha}.
\end{align*}
Observing that
\[
[\varphi]_{\alpha}\leq C  \|\oo\|_{{\rm BUC}^{\alpha}}(1+\|f'\|_{{\rm BUC}^{ \alpha}}),
\]
the latter arguments show that \eqref{Bound3} holds for $z\in\0_+$. Arguing along the same lines it is easy to see that \eqref{Bound3} is satisfied also for $z\in\0_-$.
The claim~\eqref{inftybdd} follows now from~\eqref{Bound1}-\eqref{Bound3}.
\end{proof}

We now extend in Theorem~\ref{T:Priva} Privalov's theorem to the setting considered herein where the contour integral in \eqref{F:V} is defined over an unbounded graph.
\begin{thm}\label{T:Priva}
The restrictions $v_\pm\coloneqq{}v|_{\0_\pm}$ of the function $v$  defined in \eqref{F:V} satisfy 
$${v_\pm\in {\rm BUC}^{\alpha}(\0_\pm).}$$
\end{thm}
\begin{proof}
We only establish the Hölder continuity of $v_+=:(v_+^1,v_+^2)$ (that of $v_-$ follows by using  similar arguments). We devise the proof in several steps. \medskip

\noindent{Step 1.}  Let $z,\, z'\in\ov{\0_+}$ satisfy $|z-z'|> 1/8$. Then according to Lemma~\ref{L:Plemelj} and Lemma~\ref{L:Bdd} we have
\[
|v_+(z)-v_+(z')|\leq 2\|v\|_\infty \leq 16\|v\|_\infty|z-z'|^{\alpha}\leq C|z-z'|^{\alpha}.
\]
\noindent{Step 2.}  Given $z\in\R^2,$ we set again  $d(z)\coloneqq{}{\rm dist\,}(z,\Gamma_f)$. 
Assume now that $z,\, z'\in\0_+$ are chosen such that~$|z-z'|\leq 1/8.$
Then, letting $S_{zz'}\coloneqq{}\{(1-t)z+tz'\,:\, t\in[0,1]\}$ denote the  segment that connects $z$ and $z'$
 there exists at least a point  $\ov\zeta\in S_{zz'}$ such that 
\[
d(\ov\zeta)=|\ov\zeta-\ov\zeta_\G|={\rm dist\,} (S_{zz'},\Gamma_f).
\]
We distinguish two cases. \medskip

\noindent{Step 2a.}  If $|z-z'|<|\ov\zeta-\ov\zeta_\G|$, then $S_{zz'}\subset \0_+.$ Then we have
\begin{align*}
|v_+(z)-v_+(z')|= |F(z)-F(z')|=\Big|\int_{S_{zz'}}F'(\zeta)\, d\zeta\Big|=\Big|\int_{S_{zz'}}\Big(\frac{1}{2\pi i}\int_{\Gamma_f}\frac{\varphi(\xi)}{(\xi-\zeta)^2}\, d\xi\Big)\, d\zeta\Big|,
\end{align*} 
where $F$ is the holomorphic function defined in \eqref{PHIcont}. 
Given $\zeta\in S_{zz'}$, it holds that 
\[
\int_{\Gamma_f}\frac{1}{(\xi-\zeta)^2}\, d\xi=0
\]
and therewith we get
\begin{align*}
|v_+(z)-v_+(z')|&=\Big|\int_{S_{zz'}}\Big(\frac{1}{2\pi i}\int_{\Gamma_f}\frac{\varphi(\xi)-\varphi(\zeta_\G)}{(\xi-\zeta)^2}\, d\xi\Big)\, d\zeta\Big|\\[1ex]
&\leq|z-z'|\sup_{\zeta\in S_{zz'}}\Big|\int_{\Gamma_f}\frac{\varphi(\xi)-\varphi(\zeta_\G)}{(\xi-\zeta)^2}\, d\xi\Big|\\[1ex]
&\leq |z-z'|\cdot[\varphi]_{\alpha}\sup_{\zeta\in S_{zz'}}\int_{\Gamma_f}\frac{|\xi-\zeta_\G|^{\alpha}}{|\xi-\zeta|^2}\, |d\xi|.
\end{align*} 
Recalling the definition of $\zeta_\G$, we have $|\xi-\zeta_\G|\leq |\xi-\zeta|+|\zeta-\zeta_\G|\leq 2|\xi-\zeta|$ for all $\zeta\in S_{zz'}$ and~${\xi\in\Gamma_f}$,
hence $|\xi-\zeta_\G|+|\zeta-\zeta_\G|\leq 3|\xi-\zeta|$. 
Noticing also that  $|z-z'|<|\zeta-\zeta_\G|$ for~${\zeta\in S_{zz'}}$, we obtain in view of these inequalities 
\begin{align*}
\int_{\Gamma_f}\frac{|\xi-\zeta_\G|^{\alpha}}{|\xi-\zeta|^2}\, |d\xi|&\leq 9\int_{\Gamma_f}(|\xi-\zeta_\G|+|\zeta-\zeta_\G|)^{\alpha-2}\, |d\xi|\\[1ex]
&\leq 9(1+\|f'\|_\infty)\int_{\R}(|s|+|\zeta-\zeta_\G|)^{ \alpha-2}\,  ds\\[1ex]
&\leq C |\zeta-\zeta_\G|^{\alpha-1}\leq C|z-z'|^{\alpha-1},
\end{align*} 
and therefore $|v_+(z)-v_+(z')|\leq C |z-z'|^{\alpha}. $\medskip

\noindent{Step 2b.} We now consider the second case when $|z-z'|\geq |\ov\zeta-\ov\zeta_\G|$. 
Since $\ov\zeta\in S_{zz'}$ we   have
\begin{align*}
\max\{|z-\ov\zeta_\G|,\,|z'-\ov\zeta_\G|\}\leq \max\{|z-\ov\zeta|,\,|z'-\ov\zeta|\}+|\ov\zeta-\ov\zeta_\G|\leq 2|z-z'|\leq 1/4.
\end{align*} 
Assuming there exits a constant $C>0$ such that 
\begin{align}\label{LPo}
|v_+(z)-v_+(z_0)|\leq C |z-z_0|^{\alpha}\qquad\text{$\forall\, z_0\in\Gamma_f$ and $z\in\overline{\0_+}$ with $|z-z_0|\leq1/4,$}
\end{align}
we then have 
\[
|v_+(z)-v_+(z')|\leq|v_+(z)-v_+(\ov\zeta_\G)|+|v_+(z')-v_+(\ov\zeta_\G)|\leq C\big(|z-\ov\zeta_\G|^{\alpha}+|z'-\ov\zeta_\G|^{\alpha}\big)\leq C|z-z'|^{\alpha},
\]
and the claim then follows.\medskip 

\noindent{Step 3.} It remains to establish \eqref{LPo}. Let $z_0\in\Gamma_f$ and $z\in\0_+$ satisfy  $|z-z_0|\leq1/4,$ and let~${\G=\Gamma_0+\Gamma_1}$ and $\wt\varphi $ be as defined in the proof of Lemma~\ref{L:Plemelj}.
Recalling Lemma~\ref{L:Plemelj}, it follows similarly as in Step 3 of the proof of Lemma~\ref{L:Bdd} that
\begin{align*}
|v_+(z)-v_+(z_0)|&\leq \Big|\int_{\Gamma_f-\Gamma_0}\Big(\frac{\varphi(\xi)}{\xi-z}-\frac{\varphi(\xi)}{\xi-z_0}\Big)\, d\xi\Big|+
\Big|\int_{\Gamma_1}\Big(\frac{\wt\varphi(\xi)}{\xi-z}-\frac{\wt\varphi(\xi)}{\xi-z_0}\Big)\, d\xi\Big|\\[1ex]
&\hspace{0,45cm}+\Big|\frac{1}{2\pi i}\int_{\Gamma}\frac{\wt\varphi(\xi)}{\xi-z}\, d\xi-\frac{1}{2\pi i}\PV\int_{\Gamma}\frac{\wt\varphi(\xi)}{\xi-z_0}\, d\xi-\frac{1}{2}\varphi(z_0)\Big|=:\sum_{i=1}^3T_i,
\end{align*}
with
\[
T_1+T_2\leq C|z-z_0|\leq C|z-z_0|^{\alpha}.
\]
It remains to estimate the term $T_3$ which, in view of \eqref{INTF},  can be written as 
\begin{align*}
T_3(z)&\coloneqq{}\Big|\frac{1}{2\pi i}\int_{\Gamma}\frac{\wt\varphi(\xi)}{\xi-z}\, d\xi-\frac{1}{2\pi i}\PV\int_{\Gamma}\frac{\wt\varphi(\xi)}{\xi-z_0}\, d\xi-\frac{1}{2}\varphi(z_0)\Big|\leq\sum_{i=1}^3S_i, 
\end{align*}
where, letting $z_\Gamma$ be defined by the relation $d(z)=|z-z_\G|$, we set  
\begin{align*}
S_1&\coloneqq{}| \varphi(z_\G)- \varphi(z_0)|, \qquad S_2\coloneqq{}\Big|\int_{\Gamma_1}\Big(\frac{\wt\varphi(\xi)-\wt \varphi(z_\G)}{\xi-z} - \frac{\wt\varphi(\xi)-\wt\varphi(z_0)}{\xi-z_0}\Big)\, d\xi\Big|,\\[1ex]
S_3&\coloneqq{}\Big|\int_{\Gamma_0}\Big(\frac{\varphi(\xi)-\varphi(z_\G)}{\xi-z} - \frac{\varphi(\xi)-\varphi(z_0)}{\xi-z_0}\Big)\, d\xi\Big|. 
\end{align*} 
Noticing that $|z_\G-z_0|\leq |z_\G-z|+|z-z_0|\leq 2|z-z_0|$  we obtain 
\[
S_{1} \leq[\varphi]_{\alpha}|z_\G-z_0|^{\alpha}\leq C|z-z_0|^{\alpha}.
\]
Moreover, given $\xi\in\G_1,$ we have $\min\{|\xi-z|,\, |\xi-z_0|\}\geq 3/4$ and together with \eqref{UBound}  we get
\begin{align*}
S_2&\leq\frac{16}{9}\int_{\Gamma_1} (|\wt\varphi(\xi)-\wt\varphi(z_\G)|\cdot|z-z_0|+|\varphi(z_0)- \varphi(z_\G)|) \,|d\xi|\\[1ex]
&\leq C|\G_1|(\| \varphi\|_\infty|z-z_0| +[\varphi]_{\alpha}|z_0-z_\G|^{\alpha})\leq C|z-z_0|^{\alpha}.
\end{align*}
In order to estimate $S_3$ we let $\eta\coloneqq{}|z-z_0|\in(0,1/4]$, we set  $z_0=:(x_0,f(x_0))$, and we introduce  the curve~$\Gamma_\eta\coloneqq{}\{(s,f(s))\,:\, |s-x_0|\leq 2\eta\}$.
It then holds
\[
S_3\leq S_{3a}+S_{3b}+S_{3c},
\]
where
\begin{align*}
S_{3a}&\coloneqq{}\Big|\int_{\Gamma_\eta}\Big(\frac{\varphi(\xi)-\varphi(z_\G)}{\xi-z} - \frac{\varphi(\xi)-\varphi(z_0)}{\xi-z_0}\Big)\, d\xi\Big|,\\[1ex]
S_{3b}&\coloneqq{}|z-z_0|\cdot\Big|\int_{\Gamma_0-\Gamma_\eta}\frac{\varphi(\xi)-\varphi(z_\G)}{(\xi-z)(\xi-z_0)}\, d\xi\Big|,\qquad 
S_{3c}\coloneqq{}\Big|\int_{\Gamma_0-\Gamma_\eta}\frac{\varphi(z_0)-\varphi(z_\G)}{\xi-z_0} \, d\xi\Big|.
\end{align*}
The relation $|z-z_\G|\leq |z-z_0|=\eta$  implies that $z_\G\in\G_\eta.$ Taking also into account the inequality~$|\xi-z_\G|\leq |\xi-z|+|z-z_\G|\leq 2|\xi-z|$ for~$\xi\in\Gamma_f$, we have
\begin{align*}
S_{3a}&\leq 2[\varphi]_{\alpha}\int_{\Gamma_\eta}(|\xi-z_\G|^{\alpha-1}+|\xi-z_0|^{\alpha-1})\,|d\xi|\leq  C[\varphi]_{\alpha}(1+\|f'\|_\infty)\eta^{\alpha}\leq C|z-z_0|^{\alpha}.
\end{align*}

Given $\xi\in \Gamma_0-\Gamma_\eta$,  the relation~$|\xi-z_0|\geq 2\eta=2|z-z_0|$ leads us to
\begin{equation*}
\begin{aligned}
|\xi-z|&\geq |\xi-z_0|-|z-z_0|\geq\eta=|z-z_0|,\\[1ex]
  2|\xi-z_0|&\geq |\xi-z_0|+|z_0-z|\geq|\xi-z|,\\[1ex]
3|\xi-z|&\geq |\xi-z_0|-|z-z_0|+2|z-z_0|=|\xi-z_0|+|z-z_0|.
\end{aligned}
\end{equation*}
Recalling  also  that $|\xi-z_\G| \leq 2|\xi-z|$, we then obtain
\begin{align*}
S_{3b}&\leq4|z-z_0|\cdot[\varphi]_{\alpha}\int_{\Gamma_0-\Gamma_\eta}|\xi-z|^{\alpha-2}\, |d\xi|\\[1ex]
&\leq C|z-z_0|\cdot[\varphi]_{\alpha}\int_{\Gamma_0-\Gamma_\eta}(|\xi-z_0|+|z-z_0|)^{\alpha-2}\, |d\xi|\\[1ex]
 &\leq C[\varphi]_{\alpha}(1+\|f'\|_\infty)|z-z_0|(2\eta+|z-z_0|)^{\alpha-1}\leq C|z-z_0|^{\alpha}.
\end{align*}

Finally, since $|z_0-z_\G|\leq  2|z-z_0|$, we have
\begin{align*}
S_{3c}&\leq[\varphi]_{\alpha} |z_0-z_\G|^{\alpha}\Big|\int_{\Gamma_0-\Gamma_\eta}\frac{1}{\xi-z_0} \, d\xi\Big|\leq C |z_0-z|^{\alpha}\Big|\int_{\Gamma_0-\Gamma_\eta}\frac{1}{\xi-z_0} \, d\xi\Big|
\end{align*}
and, after identifying the real and imaginary parts of the integral, we get
\begin{align*}
\Big|\int_{\Gamma_0-\Gamma_\eta}\frac{1}{\xi-z_0} \, d\xi\Big|&\leq \Big|\int_{\{2\eta\leq |s|\leq 1\}}\frac{s+f'(x_0-s)(f(x_0)-f(x_0-s))}{s^2+(f(x_0)-f(x_0-s)^2}\, ds\Big|\\[1ex]
&\hspace{0,45cm}+\Big|\int_{\{2\eta\leq |s|\leq 1\}}\frac{sf'(x_0-s)-(f(x_0)-f(x_0-s))}{s^2+(f(x_0)-f(x_0-s)^2}\, ds\Big|\\[1ex]
&\leq\Big|\frac{1}{2}\ln \Big(\frac{(2\eta)^2+(f(x_0)-f(x_0+2\eta))^2}{(2\eta)^2+(f(x_0)-f(x_0-2\eta))^2}\cdot \frac{1+(f(x_0)-f(x_0-1))^2}{1+(f(x_0)-f(x_0+1))^2}\Big)\Big|\\[1ex]
&\hspace{0.56cm}+[f']_{\alpha}\int_{\{2\eta\leq |s|\leq 1\}}|s|^{\alpha-1}\, ds\\[1ex]
&\leq  \ln (1+\|f'\|_\infty^2)+C[f']_{\alpha}.
\end{align*}
Hence, we have shown that $S_{3c}\leq C |z_0-z|^{\alpha}$
and the proof is completed.
\end{proof}

We conclude this section with the following result.
\begin{lemma}\label{L:vanis}
It holds that
\begin{align}\label{inc+iro}
\p_xv^1+\p_yv^2=0=\p_yv^1-\p_xv^2 \qquad\text{in $\R^2\setminus\Gamma_f$}
\end{align} 
and
\begin{align}\label{velova}
v_\pm(z)\to0 \qquad\text{for $z\in\0_\pm$ with $|z|\to\infty$}.
\end{align}
\end{lemma}
\begin{proof}
The relations \eqref{inc+iro} follow by direct computation.
We next prove that $v_+ $  vanishes at infinity (the claim for $v_-$ follows by arguing along the same lines).
We divide the proof in two steps.\medskip

\noindent{Step 1.} We first show that $v_+(x,f(x))\to 0$ for $|x|\to\infty.$ 
Recalling Lemma~\ref{L:Plemelj} and~\eqref{Bom}, we write
\[
v_+(x,f(x))=\frac{1}{2\pi }\big(-B_{1,1}^0(f)[\oo],B_{0,1}^0(f)[\oo]\big)(x)-\frac{1}{2}\frac{\oo(1,f')}{1+f'^2}(x),\qquad x\in\R.
\] 
Because $\oo\in {\rm BUC}^{\alpha}(\R)\cap L_p(\R)$, the last term on the right vanishes at infinity.
We next prove that, given $n,\, m\in\N$,  we also have
\begin{align}\label{bnmvi}
B_{n,m}^0(f)[\oo](x)\to0\qquad\text{for $|x|\to\infty$.}
\end{align}
Let thus $\e>0$ be given and choose $N>0$ such that 
\[
\|f'\|_\infty^n\|\oo\|_p\Big(\frac{2}{(p'-1)N^{p'-1}}\Big)^{1/p'}\leq\frac{\e }{2 },
\]
where $p'$ is the adjoint exponent to $p$.
This choice together with H\"older's inequality then yields
\begin{align*}
|B_{n,m}^0(f)[\oo](x)|\leq T(x)+\|f'\|_\infty^n\|\oo\|_p\Big(\frac{2}{(p'-1)N^{p'-1}}\Big)^{1/p'}\leq T(x)+\frac{\e}{2},
\end{align*}
where
\begin{align*}
 T(x)\coloneqq{}\Big|\PV\int_{\{|s|\leq N\}} \cfrac{\big(\delta_{[x,s]} f / s\big)^n}{\big[1+\big(\delta_{[ x,s]}f / s\big)^2\big]^m} \frac{\oo(x- s)}{ s}\, d s\Big|,\qquad x\in\R.
\end{align*}
In order to estimate $T(x)$ we note that
\begin{align*}
T(x)\leq C\int_0^N \Big|\frac{\oo(x- s)-\oo(x+s)}{ s}\Big|+ |\oo(x+s)|\cdot\Big|\frac{f(x+s)-2f(x)+f(x-s)}{s^2}\Big|\, d s,
\end{align*}
with $C$ depending only on $n$, $m$, and $\|f'\|_\infty$.
Taking into account that $\oo$ vanishes at infinity we obtain for $|x|>M$, where $M>N$ is chosen sufficiently large, that 
\[
T(x)\leq C\big([\oo]_\alpha^{1/2}\|\oo\|_{L_\infty(\{|x|>M-N\})}^{1/2}+\|\oo\|_{L_\infty(\{|x|>M-N\})}\big)\leq \frac{\e}{2}.
\]
This establishes \eqref{bnmvi}.\medskip

\noindent{Step 2.} We now prove  that $v_+(z)\to 0$ for $|z|\to\infty.$
Let thus $\e>0$ be given. From Step~1 we find $x_0>0$ such that $|v_+(x,f(x))|\leq \e/2$ for all $|x|\geq x_0.$
Given $z=(x,y)\in\0_+$, let again~${d(z)\coloneqq{}{\rm dist \,}(z,\Gamma_f)=|z- z_\G|}$ with $z_\G\in\G$.

Assume first that $d(z)\leq \delta\coloneqq{} \min\{1,\e/(2(1+[v_+]_{\alpha}))\}$.
Let $x_1\coloneqq{}x_0+1$. 
If   $z=(x,y)\in\0_+$  satisfies with  $d(z)\leq \delta$ and~${|x|\geq x_1}$, we deduce for the corresponding point~${z_\G\coloneqq{}(x_\G,f(x_\G))}$ that $|x_\G|\geq x_0$. 
Hence, for all such~$z\in\0_+$,      Theorem~\ref{T:Priva} leads us to
\begin{align*}
|v_+(z)|=|v_+(z)-v_+(z_\G)|+|v_+(z_\G)|\leq [v_+]_{\alpha}d(z)+\e/2\leq\e.
\end{align*}

Assume now that $d(z)\geq \delta$.
Let  $s_0>0$  be chosen such that 
\[
\|\oo\|_p\Big(\frac{2}{(p'-1)s_0^{p'-1}}\Big)^{1/p'} \leq\frac{\e}{2}.
\]
It then holds
 \begin{align*}
 |v_+(z)|&\leq  \int_\R\frac{|\oo(x-s)|}{\sqrt{s^2+(y-f(x-s))^2}}\, ds\leq T(z)+  \int_{\{|s|> s_0\}}\frac{|\oo(x-s)|}{|s|}\, ds  \\[1ex]
 &\leq  T(z)+  \|\oo\|_p\Big(\frac{2}{(p'-1)N^{p'-1}}\Big)^{1/p'}\leq  T(z)+  \frac{\e}{2},
 \end{align*}
 where 
 \[
T(z)\coloneqq{}\int_{\{|s|< s_0\}}\frac{|\oo(x-s)|}{\sqrt{s^2+(y-f(x-s))^2}}\, ds,\qquad z=(x,y)\in\0_+,\, d(z)\geq \delta.  
 \]
Let  $N>0$ be chosen such that 
$$\frac{4s_0\|\oo\|_\infty}{N}+\frac{2s_0 \|\oo\|_{L_\infty(\{|x|\geq N\})}}{\delta}\leq\frac{\e}{2} $$
and set $M_1\coloneqq  N+s_0,  $   $M_2\coloneqq N+2\|f\|_{L_\infty(\{|x|\leq M_1+s_0\})}$, and $M\coloneqq 2\max\{M_1,\, M_2\}.$ 
Given~${|z|\geq M}$, we distinguish two cases.

(1)\, If~$|x|\geq M_1,$ then
\[
T(z)\leq \frac{2s_0}{\delta}\|\oo\|_{L_\infty(\{|x|\geq M_1-s_0\})}=\frac{2s_0}{\delta}\|\oo\|_{L_\infty(\{|x|\geq N\})}\leq\frac{\e}{2}.
\]

(2)\, If $|x_1|\leq M_1$ and $|y|\geq M_2$, then $|y-f(x-s)|\geq |y/2|$ and therefore
 \[
T(z)\leq \frac{4s_0}{|y|}\|\oo\|_\infty\leq\frac{4s_0}{N}\|\oo\|_\infty\leq \frac{\e}{2}.
\]
Hence $|v_+(z)|\leq\e$ for all  $z\in\0_+$ that satisfy~${d(z)\geq \delta}$ and~${|z|\geq M}$.

To summarize,  for all $z\in\0_+$ with $|z|\geq \max\{M, x_1+\|f\|_{L_\infty(\{|x|\leq x_1+1\})}+1\} $ we have established that~${|v_+(z)|\leq\e}$ and this completes the proof.
\end{proof}

\subsection*{Acknowledgement}
The authors gratefully acknowledge the support by the RTG 2339
 ``Interfaces, Complex Structures, and Singular Limits'' of the German Science Foundation (DFG).

\bibliographystyle{siam}
\bibliography{JB}
\end{document}